\documentclass{amsart}
\usepackage{stability-arXiv}
\usepackage{amsrefs,enumerate}
\newcommand\xcite[1]{\def\xtopic{#1}\cite{#1}}
\author[DeBacker]{Stephen DeBacker}
\address{University of Michigan \\ Ann Arbor, MI \ \ 48109}
\email{smdbackr@umich.edu}
\author[Spice]{Loren Spice}
\address{Texas Christian University \\ Fort Worth, TX \ \ 76129}
\email{l.spice@tcu.edu}
\thanks{The authors were both partially supported by
National Science Foundation Focussed-Research Grant 0854897.
The second-named author was partially supported by
Simons Foundation Collaboration Grant 246066.}

\title{Stability of character sums for positive-depth,
supercuspidal representations}
\begin{document}
\begin{abstract}
We re-write the character formul{\ae} of
\cite{adler-spice:explicit-chars} in a form amenable to
explicit computations in \(p\)-adic harmonic analysis, and
use them to prove the stability of character sums for a
modification of Reeder's conjectural positive-depth,
unramified, toral supercuspidal L-packets
\cite{reeder:sc-pos-depth}.
\end{abstract}
\maketitle

\section{Introduction}

In \cite{reeder:sc-pos-depth}*{\S6.6, p.~18}, Reeder
generalises his joint work with the first-named author
on the depth-\(0\), supercuspidal local Langlands
correspondence
\cite{debacker-reeder:depth-zero-sc}*{Theorem 4.5.3}
to construct some candidates for
positive-depth, supercuspidal L-packets.
These consist of what could be called the unramified, toral
supercuspidal representations.
(In the notation of
Definitions \ref{defn:cusp-pair} and \ref{defn:cusp-ind},
they are the representations \(\YuUp_{G'}^G \pi'\)
arising from cuspidal pairs \((\bG', \pi')\) in which
\(\bG'\) is an unramified torus.)
He proves that these sets satisfy many of the
properties expected of L-packets, but does not show that
appropriate combinations of the characters in an L-packet
are stable.  In this paper, we use the
character formul{\ae} of \cite{adler-spice:explicit-chars}
to demonstrate the appropriate stability,
after replacing the construction of Adler--Yu used in
\cite{reeder:sc-pos-depth}*{\S3} by an appropriate `twisted'
analogue (Definition \ref{defn:twisted-cusp-ind}).
We show that there are cases when this twist is necessary
(Example \ref{example:must-twist}).
It is possible, though we have not yet explored this, that
our twist bears some relation to the notion of
\textit{rectifier} introduced by Bushnell and Henniart
\cite{bushnell-henniart:tame-llc-3}*{Definition 1}.
In any case, recent work of Yu \cite{yu:supercuspidal-revisited}
suggests that a different lift of the
Weil representation in the construction of
\cite{yu:supercuspidal}*{Proposition 4.6 and Theorem 11.5}
would obviate the need for this \textit{ad hoc} twist.

The paper proceeds as follows.  In \S\ref{sec:notation}, we
establish some basic notation regarding the groups and
fields that we consider.
Of particular importance in this section is the
`reduced discriminant' (Definition \ref{defn:disc}), which
we use to normalise characters and orbital integrals in
Definition \ref{defn:normal-harm}.
In \S\ref{sec:MP}, we recall the definitions of the
Moy--Prasad filtrations of a \(p\)-adic group and its Lie
algebra, and introduce some refinements
(\S\ref{sec:MP-mixed-depth}).

The meat of the paper is \S\ref{sec:ratl}, in which we
re-cast the character formul{\ae} of
\cite{adler-spice:explicit-chars} in a form that is more
suitable for our use and, we believe, for later explicit
calculations in harmonic analysis.
In this section, we handle a wide range of characters, not
just the toral ones considered by Reeder; but our work
inherits from \cite{adler-spice:explicit-chars}
a `compactness' condition
(see Definition \ref{defn:cusp-pair}).
We believe that Theorem \ref{thm:ratl} provides
an idea of the shape of a formula for the character of an
arbitrary tame supercuspidal representation, even one that does
not satisfy the compactness condition;
verifying this will be the subject of future work.

To further explain the content of \S\ref{sec:ratl}, we note
that the character formul{\ae} of
\cite{adler-spice:explicit-chars} are explicit, in the sense
that all ingredients are described in terms of the
parametrising data (see Definition \ref{defn:cusp-ind});
but they can be difficult to evaluate in particular cases.
(See, for example,
\xcite{adler-debacker-sally-spice:sl2}*{\S\xref{sec:ordinary}}
for the case \(G = \SL_2(\field)\).)
There are four main ingredients:
	\begin{itemize}
	\item The character of \(\pi'\).  We assume
inductively that this is already known.  The base case of
this induction is provided by depth-\(0\) character
formul\ae; see
\cite{debacker-reeder:depth-zero-sc}*{Lemma 10.0.4}.
	\item A product of indices of subquotients (see
\xcite{adler-spice:explicit-chars}*
	{Proposition \xref{prop:induction1}}).
This is discussed in \S\ref{sec:indices},
and computed in Corollary \ref{cor:const}.
	\item A fourth root of unity
(see \xcite{adler-spice:explicit-chars}*
	{Propositions \xref{prop:theta-tilde-phi}
	and \xref{prop:gauss-sum}}).
This is discussed in \S\ref{sec:root},
and computed (or, at least, re-cast in a form more amenable
to stability calculations)
in Proposition \ref{prop:root}.
Its behaviour is rather subtle; see
Proposition \ref{prop:stable-sign} in the unramified case,
and \cite{spice:signs-W} for the effects of ramification.
	\item The Fourier transform of an orbital integral on a
centraliser subgroup.
This contribution is predicted by Murnaghan--Kirillov theory
(see, for example, \cite{jkim-murnaghan:charexp}),
and plays the r\^ole of a Green's function in the
Deligne--Lusztig character formul{\ae}
(see, for example, \cite{debacker-reeder:depth-zero-sc}*
	{Lemma 12.4.3}).
We treat it as an uncomputed `primitive ingredient',
although the normalisation of the measure occurring in its
definition is quite important.
See \S\ref{sec:orbital}, particularly
Proposition \ref{prop:orbit}.
	\end{itemize}

Finally, in \S\ref{sec:stable}, we use the character formula
Theorem \ref{thm:ratl} to show stability of character sums.
Here we restrict our attention to toral supercuspidal
representations, and, for most of the section (with the
exception of Corollary \ref{cor:ratl}), even to
\emph{unramified}, toral supercuspidal representations.
Motivated by Theorem \ref{thm:ratl},
we begin by introducing the modified Yu-type construction
(Definition \ref{defn:twisted-cusp-ind})
mentioned above.
This simplifies the character formula slightly
(Corollary \ref{cor:ratl}),
and allows us to write an analogue
(Theorem \ref{thm:almost-stable}) of
\cite{debacker-reeder:depth-zero-sc}*{Lemma 11.0.2}
that expresses the sums of the characters in a
Reeder-type L-packet in terms of Fourier transforms
of stable orbital integrals.
Finally, in Theorem \ref{thm:stable}, we use this
expression, together with work of Waldspurger on such
Fourier transforms (\cite{waldspurger:transfert}*{Th\'eor\`eme 1.5}
and \cite{debacker-reeder:depth-zero-sc}*{Lemma 12.2.3}),
to prove the desired stability.

\numberwithin{thm}{subsection}

\section{Notation}
\label{sec:notation}

\subsection{Fields and characters}

\begin{defn}
Let \mterm\field be a field
and \mterm\ord a non-trivial discrete valuation on \field
(with value group \Z)
with respect to which \field is complete.
Choose a fixed separable closure \(\mterm\sepfield/\field\),
and denote the unique extension of \(\ord\) to \sepfield
again by \(\ord\).
Let \mterm\unfield (respectively, \mterm\tamefield) be the
maximal unramified (respectively, tame) extension of \field
inside \sepfield.
Put \(\mterm\Gamma = \Gal(\sepfield/\field)\)
and
\(\mterm{\Gamma\unfix} = \Gal(\sepfield/\unfield)\).

If \(E/\field\) is a separable extension, then denote by
\begin{itemize}
\item
\mterm{\pint_E} the ring of integers of \(E\),
\item
\mterm{\pp_E} the unique maximal ideal in \(\pint_E\),
\item
\mterm{\ff_E} the residue class field \(\pint_E/\pp_E\)
of \(E\),
and
\item
\mterm{q_E} the cardinality of \(\ff_E\).
\end{itemize}
We sometimes drop the superscript \(E\) when \(E = \field\).
We require that \ff be finite (equivalently, that \field be
locally compact).

Finally,
let \(\Frob\) be any element of \(\Gamma\)
whose image in \(\Gal(\ff_\unfield/\ff)\) is the inverse of
\anonmapto t{t^q}.
\end{defn}

\begin{defn}
\label{defn:add-char}
Let \AddChar be an additive character of \field,
\ie, a homomorphism \anonmap\field{\C\mult},
chosen as in
\xcite{adler-spice:explicit-chars}*
	{\S\xref{sec:generalities}}.
In particular, \AddChar is trivial on \pp but not on \pint.
Write \AddChar again for the resulting character
of \ff.
Given a finite-dimensional \field- or \ff-vector space
\(V\),
these characters fix a canonical identification of the
linear and Pontrjagin duals of \(V\).
\end{defn}

\begin{defn}
\label{defn:FT}
If \(V\) is a finite-dimensional \field- or \ff-vector
space, with Pontrjagin dual \(\widehat V\),
and \map f V\C is smooth
(that is, locally constant with respect to the
analytic topology on \(V\)---which is no condition if the
field of scalars is \ff),
then the Fourier transform \map{\hat f}{\widehat V}\C
of \(f\) is defined by
\[
\hat f(\chi) = \int_V f(X)\chi(X)\textup dX
\qforall{\(\chi \in \widehat V\),}
\]
where \(\textup dX\) is a suitably normalised Haar measure
on \(V\) (see below).
The choice of additive character \AddChar
(Definition \ref{defn:add-char})
allows us to regard \(\hat f\) as a function on the
linear dual \(V^*\).
Since \(\widehat{\widehat V}\) is canonically identified with \(V\),
we may, and do, further regard \(\Hat{\Hat f}\) as a function on
\(V\).
Then the Haar measure \(\textup dX\) is chosen so that the
fourth power of the Fourier transform is the identity.
\end{defn}

\subsection{Tori and related functions}
\label{sec:tori}

For the remainder of this paper, let \mterm\bG be a
connected, reductive \field-group.

Whenever a \field-group is denoted by a boldface Roman
letter, such as \bH, we denote its Lie algebra by the
corresponding boldface Fraktur letter, such as \pmb\fh.
Whenever a \field-variety is denoted by a boldface letter,
such as \bY, we denote its set of rational points by
the corresponding non-boldface letter, such as \(Y\).
Thus, for example, \fh is the space of \field-rational
points of the Lie algebra of \bH.
The identity component of \bH is denoted by \(\bH\conn\),
and, by abuse of notation, the group of \field-rational
points of \(\bH\conn\) by \(H\conn\) (even though this
latter group is actually totally disconnected in the
analytic topology inherited from \field).

A maximal torus \bT in \bG (that is defined over \field)
is called a \term{maximal \(G\)-torus}.

\begin{defn}
\label{defn:tori}
If \bT is a maximal \(G\)-torus, then
write \mterm{\wtilde\Root(\bG, \bT)}
for the set of weights of the adjoint action of
\(\bT_\sepfield\) on \(\pmb\fg_\sepfield\),
and
\(\mterm{\Root(\bG, \bT)}
= \wtilde\Root(\bG, \bT) \setminus \sset0\)
for the absolute root system of \bT in \bG.
\end{defn}

\begin{defn}
An element \(\gamma \in G\)
(respectively, \(X \in \fg\);
respectively, \(X^* \in \fg^*\))
for which the identity component of
\(C_\bG(\gamma)\)
(respectively, the stabiliser of \(X\) in the adjoint action;
respectively, the stabiliser of \(X^*\) in the coadjoint action)
is a torus is called \term{regular semisimple}.
Write \mterm{G\rss}
(respectively, \mterm{\fg\rss};
respectively, \mterm{\fg^{*\,\textup{rss}}})
for the set of such elements.
A regular, semisimple element with connected centraliser is
called \term{strongly regular semisimple}
\cite{debacker-reeder:depth-zero-sc}*{\S2.9, p.~817}.
\end{defn}

For the remainder of this section, fix a maximal \(G\)-torus
\bT, and put \(\Root = \Root(\bG, \bT)\)
and \(\wtilde\Root = \wtilde\Root(\bG, \bT)\).

As in \cite{yu:supercuspidal}*{\S4, p.~591},
we identify the Lie algebra \ft and its dual \(\ft^*\) with
the subgroups of \fg and \(\fg^*\), respectively,
fixed by the (adjoint or coadjoint, respectively) action of \(T\).

The terminology and symbols below
(Definitions \ref{defn:symm} and \ref{defn:root-constants})
make sense for arbitrary characters, but we need them only
for roots.
The notation is as in
\xcite{adler-spice:explicit-chars}*
	{Notation \xref{notn:root-constants}},
except that we write \(\field_\alpha\) in place of
\(F_\alpha\), and \(\sigma_\alpha\) in place of
\(\eta_\alpha\).

\begin{defn}
\label{defn:symm}
If \(\alpha \in \Root\), then \(\alpha\) is
\term{symmetric} exactly when its \(\Gamma\)-orbit contains
\(-\alpha\) (and non-symmetric otherwise).
A symmetric root is \term{ramified} exactly when its
\(\Gamma\unfix\)-orbit contains \(-\alpha\)
(and \term{unramified} otherwise).

If \(S\) is any subset of \Root, then
\mterm{S\symm},
\mterm{S\symmunram},
\mterm{S\symmram},
and \mterm{S\nosymm}
are the sets of roots in \(S\) that are
symmetric,
symmetric and unramified,
symmetric and ramified,
and non-symmetric,
respectively.
If \(\alpha \in \Root\symm\),
then let \mterm{\sigma_\alpha}
be any element of \(\Gamma\unfix\)
(if \(\alpha\) is ramified)
or of \(\Gamma\) (if \(\alpha\) is unramified)
such that \(\sigma_\alpha\alpha = -\alpha\).
\end{defn}

\begin{defn}
\label{defn:root-constants}
If \(\alpha \in \wtilde\Root\), then put
\(\mterm{\Gamma_\alpha} = \stab_\Gamma \alpha\)
and
\(\mterm{\Gamma_{\pm\alpha}} = \stab_\Gamma \sset{\pm\alpha}
	= \langle\Gamma_\alpha, \sigma_\alpha\rangle\).
Write \mterm{\field_\alpha} for the fixed field in \sepfield
of \(\Gamma_\alpha\), and similarly for
\mterm{\field_{\pm\alpha}}.
All our notations below with subscript \(\alpha\)
have obvious analogues with subscript \(\pm\alpha\),
which we use as necessary.
Write
\def\pmterm#1{%
	\index{notation}{\(#1_{\pm\alpha}\)}%
	\mterm{#1_{\alpha}}%
}
	\begin{itemize}
	\item
\(\pmterm\pint = \pint_{\field_\alpha}\)
and
\(\pmterm\pp = \pp_{\field_\alpha}\);
	\item \pmterm n, \pmterm e, and \pmterm f
for the degree, ramification degree,
and residual degree, respectively, of \(\field_\alpha/\field\);
	\item \(\pmterm\ff = \ff_{\field_\alpha}\),
and \(\pmterm q = q_{\field_\alpha}\);
	\item if \(\alpha\) is symmetric and unramified,
then \mterm{\ff_\alpha^1} is the kernel of the norm map from
\(\ff_\alpha\) to \(\ff_{\pm\alpha}\);
and
	\item
\(\pmterm\Z = \ord(\field_\alpha\mult) = e_\alpha\inv\Z\)
for the value group of \(\field_\alpha\).
	\end{itemize}
These objects depend only on the \(\Gamma\)-orbit
\(\omega\) of \(\alpha\); so we denote them instead by
\(\field_\omega\), \&c., when convenient.
\end{defn}

\begin{defn}
If \(\alpha \in \wtilde\Root\), then
\(\pmb\fg_\alpha\) is the \(\alpha\)-weight space for the
action
of \(\bT_{\field_\alpha}\) on \(\pmb\fg_{\field_\alpha}\),
and \mterm{\bG_\alpha} is the corresponding subgroup of
\(\bG_{\field_\alpha}\)
(denoted by \(U_\alpha\) in
\cite{springer:lag}*{Proposition 8.1.1(i)}
when \(\alpha \ne 0\)).
Note that these are defined only over \(\field_\alpha\), not
necessarily over \field.
\end{defn}

\begin{rem}
In particular, \(\mterm{\fg_0} = \ft\)
and \(\mterm{\bG_0} = \bT\), both of which are defined over
\(\mterm{\field_0} = \field\).
\end{rem}

\begin{defn}
\label{defn:ord-by-elt}
For \(\gamma \in T\) and \(X \in \ft\), and any character
\(\chi\) of \(\bT_\sepfield\), put
\[
\mterm{\ord_\gamma} \chi = \ord\bigl(\chi(\gamma) - 1\bigr)
\qandq
\ord_X \chi = \ord \textup d\chi(X).
\]
For \(X^* \in \ft^*\) and \(\alpha \in \Root\),
define
\(\mterm{\textup d\alpha^\vee(X^*)}
\ldef X^*(\textup d\alpha^\vee(1))\)
and put
\(\mterm{\ord_{X^*}} \alpha
= \ord \textup d\alpha^\vee(X^*)\).
\end{defn}

Note that \(\ord_\gamma\) is a non-negative function
exactly when \(\gamma\) is bounded
(that is, has bounded orbits on the set \(\BB(\bG, \field)\)
of Definition \ref{defn:building} below,
in the sense of
\cite{bruhat-tits:reductive-groups-1}*
	{Exemple 3.1.2(b)}).

It is well known that the discriminant function on \(G\)
(see, for example, \cite{hc:harmonic}*{\S VI.1, p.~63})
controls the `blow up' of characters
\cite{hc:harmonic}*{Theorem VI.8.14}.
We introduce a slight variant where, rather than taking the
\emph{globally} lowest coefficient in a characteristic
polynomial (which may vanish), we take the \emph{pointwise}
lowest coefficient (which never vanishes).
It seems that this may actually provide a better measure of
the blow up of the character; at least, it gives us a more
natural inductive formula for characters
(Theorem \ref{thm:ratl}).

\begin{defn}
\label{defn:disc}
If \(\gamma \in G\)
(respectively, \(X \in \fg\)),
\index{notation}{D_G!group or Lie algebra@\(\redD_G\)}%
then \(\redD_G(\gamma)\)
(respectively, \(\redD_G(X)\))
is the coefficient of the lowest-degree power of \(t\) in
\(\det(\Ad(\gamma) - 1 + t)\)
(respectively, \(\det(\ad(X) + t)\)).
\end{defn}

\begin{rem}
\label{rem:disc:roots}
Preserve the notation of Definition \ref{defn:disc}.
If \(\gamma \in G\) and \(X \in \fg\) have Jordan
decompositions
\(\gamma = \gamma\semi\dotm\gamma\unip\)
and
\(X = X\semi + X\textsub{nil}\), respectively, then
\(\redD_G(\gamma) = \redD_G(\gamma\semi)\)
and
\(\redD_G(X) = \redD_G(X\semi)\).
Now suppose that \(\gamma \in T\) and \(X \in \ft\).
By passing to a splitting field for \bT, we see that
\[
\redD_G(\gamma)
= \prod_{\substack{\alpha \in \Root \\ \alpha(\gamma) \ne 1}}
	\bigl(\alpha(\gamma) - 1\bigr)
\qandq
\redD_G(X)
= \prod_{\substack{\alpha \in \Root \\ \textup d\alpha(X) \ne 0}}
	\textup d\alpha(X),
\]
hence that
\[
\abs{\redD_G(\gamma)}
= \prod_{\substack{\alpha \in \Root \\ \alpha(\gamma) \ne 1}}
	q^{-\ord_\gamma \alpha}
\qandq
\abs{\redD_G(X)}
= \prod_{\substack{\alpha \in \Root \\ \textup d\alpha(X) \ne 0}}
	q^{-\ord_X \alpha}.
\]
It is therefore natural to define
\[
\redD_G(X^*)
= \prod_{\substack{\alpha \in \Root \\ \textup d\alpha^\vee(X^*) \ne 0}}
	\textup d\alpha^\vee(X^*)
\qforall{\(X^* \in \ft^*\).}
\]
\end{rem}

\section{Moy--Prasad filtrations}
\label{sec:MP}

As in \S\ref{sec:tori}, let \bG be a connected,
reductive \field-group.

\begin{defn}
Put \(\tR = \R \sqcup \set{{r+}}{r \in \R}\).
This carries a natural order,
and an order-reversing involution \(\wtilde\cdot\)
defined by \(\tilde r = {(-r)+}\)
and \(\widetilde{r+} = -r\) for \(r \in \R\).
\end{defn}

Note that, unlike
\xcite{adler-spice:good-expansions}*
	{\S\xref{sec:notation}},
we do not find it convenient to include \(+\infty\) in \tR.

\begin{defn}
\label{defn:building}
Write
\mterm{\BB(\bG, \field)} for the enlarged,
and \mterm{\rBB(\bG, \field)} for the reduced,
Bruhat--Tits building of \bG over \field
(\cite{bruhat-tits:reductive-groups-1}*
	{D\'efinition 7.4.2}
and \cite{bruhat-tits:reductive-groups-2}*
	{\S4.2.16}).
There is a natural projection
\anonmap{\BB(\bG, \field)}{\rBB(\bG, \field)},
denoted by \anonmapto x{\mterm{\ol x}}.
\end{defn}

\begin{rem}
\label{rem:compatibly-filtered}
If \bT is a \tamefield-split, maximal \(G\)-torus, then
there is a family of embeddings of \(\BB(\bT, \field)\) in
\(\BB(\bG, \field)\), with canonical image.
(To see this, regard the apartment \(\AA(\bT_E, E)\) of
\(\bT_E\) as a subsimplex of \(\BB(\bG, E)\),
where \(E/\field\) is the splitting field of \bT,
and then realise \(\BB(\bT, \field)\)
and \(\BB(\bG, \field)\) as the sets of
\(\Gamma\)-fixed points of \(\AA(\bT_E, E)\)
and \(\BB(\bG, E)\) \cite{tits:corvallis}*{\S2.6.1}.)
\end{rem}

For the remainder of this section, fix a \tamefield-split,
maximal \(G\)-torus \bT
and a point \(x \in \BB(\bG, \field)\) lying in
the image of \(\BB(\bT, \field)\)
(which we abbreviate to just ``\(x \in \BB(\bT, \field)\)'').
Put \(\Root = \Root(\bG, \bT)\) and
\(\wtilde\Root = \wtilde\Root(\bG, \bT)\).

\subsection{Valuations of root data}

\begin{defn}
\label{defn:MP}
If \(r \in \tR\), then write
\[
\mterm{\fg_{x, r}}\text,\quad
\mterm{\fg^*_{x, r}}\text,
\qandq
\text{\mterm{G_{x, r}} (if \(r \ge 0\))}
\]
for the associated Moy--Prasad filtration subgroups of
\fg, \(\fg^*\), and \(G\), respectively
(see \cite{moy-prasad:k-types}*{\S\S2.6, 3.2, 3.5}
and \cite{moy-prasad:jacquet}*{\S\S3.2, 3.3}).
Also, put
\(\mterm{G_r}
= \bigcup_{y \in \BB(\bG, \field)} G_{y, r}\)
and
\(\mterm{\fg_r}
= \bigcup_{y \in \BB(\bG, \field)} \fg_{y, r}\).
\end{defn}

In particular, \(G_{x, 0}\) is the parahoric subgroup of
\(G\) associated to \(x\)
\cite{bruhat-tits:reductive-groups-2}*{D\'efinition 5.2.6},
and
\(\fg^*_{x, r} = \sett
	{X^* \in \fg^*}
	{\(\langle X^*, Y\rangle \in \pp\)
		for all \(Y \in \fg_{x, \tilde r}\)}\).

If \mc G is a group equipped with a filtration
\((\mc G_i)_{i \in I}\) by subgroups,
then we frequently write \mterm{\mc G_{i:j}} in place of
\(\mc G_i/\mc G_j\) when \(\mc G_j \subseteq \mc G_i\)
(even if the quotient is not a group).
Thus, for example, \(G_{x, 0:{0+}}\) is the group of
\ff-points of a reductive \ff-group
\cite{moy-prasad:k-types}*{\S3.2, p.~398}.

Our (rational) character computation Theorem \ref{thm:ratl}
requires the existence of a `mock-exponential' map that
behaves, in many ways, like an exponential map,
but may sacrifice some equivariance in order to converge on
a larger domain.
Hypothesis \ref{hyp:stronger-mock-exp} below,
which depends implicitly on a positive rational number \(r\),
formalises this.
Note, however, that we do not need this hypothesis until
Theorem \ref{thm:ratl} (at which point \(r\) stands for
the depth of the representation whose character we compute
there).
We also require the hypothesis for the consequences of
Theorem \ref{thm:ratl} in \S\ref{sec:stable}; but, in that
section, we actually impose the \emph{stronger}
hypothesis
\cite{debacker-reeder:depth-zero-sc}*{Restriction 12.4.1(2)}.

\begin{hyp}
\label{hyp:stronger-mock-exp}
There exists a family of bijections
\[
\bigl(
	\map{\mterm{\mexp_{\bS, \ol y}}}{\fg_{y, r}}{G_{y, r}}
\bigr)_{\substack{
	\text{\bS a \tamefield-split maximal \(G\)-torus} \\
	\ol y \in \im \anonmap{\BB(\bS, \field)}{\rBB(\bG, \field)}
}}
\]
satisfying \xcite{adler-spice:explicit-chars}*
	{Hypotheses \xref{hyp:mock-exp}
	and \xref{hyp:strong-mock-exp}},
\emph{and}, whenever
	\begin{itemize}
	\item \bH is a reductive subgroup of \bG,
	\item \bS is a \tamefield-split maximal \(G\)-torus
that is contained in \bH,
	\item \(y \in \BB(\bS, \field)\),
	and
	\item \(Y \in \fh_{x, r}\),
	\end{itemize}
also the following conditions:
\begin{enumerate}
\item\label{hyp:stronger-mock-exp:disc}
\(\smabs{\redD_H(Y)}
= \smabs{\redD_H(\mexp_{\bS, \ol y} Y)}\),
and
\item\label{hyp:stronger-mock-exp:equi}
if \(g \in G\) and \(\Int(g)\mexp_{\bS, \ol y}Y \in G_{y, r}\),
then
\[
\hat\mu^{\Int(g)H}_{\Ad^*(g)X^*}
	(\mexp_{\bS, \ox}\inv \Int(g)\mexp_{\bS, \ox}Y)
= \hat\mu^H_{X^*}(\Ad(g)Y)
\quad\text{for all \(X^* \in \fh^*_{x, -r}\),}
\]
where \(\hat\mu\) is as in Definition \ref{defn:normal-harm}.
\end{enumerate}
\end{hyp}

The point \(x \in \BB(\bT, \field)\) defines a
`valuation of root datum', roughly in the sense of
\cite{bruhat-tits:reductive-groups-1}*{\S6.1.1}, except that
we must take into account the fact that \bT need not be
\field-split (nor even maximally \field-split).

\begin{defn}
\label{defn:ord-x}
For \(\alpha \in \wtilde\Root\) and \(i \in \wtilde\R\),
write
\(\pmb\fg_\alpha(\field_\alpha)_{x, i}
\ldef \pmb\fg_\alpha(\field_\alpha) \cap
\pmb\fg(\field_\alpha)_{x, i}\),
and put
\(\mterm{\ord_x} \alpha = \set
	{i \in \R}
	{\pmb\fg_\alpha(\field_\alpha)_{x, i:{i+}} \ne 0}\).
\end{defn}

\begin{rem}
\label{rem:ord-unram}
Note that \(\ord_x \alpha\) implicitly depends on the
underlying field \field (and on the normalisation of its
valuation).
Since \(\unwrap\unfield_\alpha/\field_\alpha\) is unramified,
we have that
\(\pmb\fg(\field_\alpha)_{x, i:{i+}}
= \pmb\fg(\unwrap\unfield_\alpha)_{x, i:{i+}}
	^{\Gamma_\alpha}\)
\cite{yu:supercuspidal}*{Corollary 2.3},
hence,
by Hilbert's Theorem 90
\cite{serre:galois}*{Proposition II.1.2.1},
that
\(\pmb\fg(\field_\alpha)_{x, i:{i+}} \ne 0\)
if and only if \(\pmb\fg(\unwrap\unfield_\alpha)_{x, i:{i+}} \ne 0\).
In other words, \(\ord_x \alpha\) does not change if \field
is replaced by an unramified extension.
\end{rem}

\begin{rem}
Since \(\pmb\fg_\alpha(\field_\alpha)\) is a
\(1\)-dimensional \(\field_\alpha\)-space, we have
that \(\ord_x \alpha\) is a \(\Z_\alpha\)-torsor.
With the terminology of \cite{moy-prasad:k-types}*{\S2.5},
\[
\ord_x \alpha = \sett
	{\psi(x)}
	{\(\psi\) is an affine \(\field_\alpha\)-root
		with gradient \(\alpha\)}.
\]
\end{rem}

Lemma \ref{lem:opp-ord} below is of interest mainly because
of its consequences,
Corollaries \ref{cor:opp-ord} and \ref{cor:ord-symm}.

\begin{lem}
\label{lem:opp-ord}
If \(\alpha \in \Root\) and \(i, j \in \tR\)
with \(i \le j\),
then the pairing
\[
\pmb\fg_\alpha(\field_\alpha)_{x, i:j}
	\otimes \pmb\fg_{-\alpha}(\field_\alpha)_{x, \tilde\jmath:\tilde\imath}
\xrightarrow{[\cdot, \cdot]} \pmb\ft(\field_\alpha)_{0:{0+}}
\xrightarrow\alpha \ff_\alpha
\]
is non-degenerate.
\end{lem}

\begin{proof}
Let \(\bG_{\pm\alpha}\) be the \(\field_\alpha\)-split,
rank-\(1\), semisimple Lie group generated by the
(\(\pm\alpha\))-root subgroups of \bG.  
By \cite{springer:lag}*{\S16.3.1}, the proof of
\cite{springer:lag}*{Theorem 7.2.4} goes through to show
that there is a \(\field_\alpha\)-surjection
\anonmap{\bG_{\pm\alpha}}{\PGL_2}.
(Alternatively, a more machinery-heavy proof could apply the
isomorphism theorem \cite{springer:lag}*{Theorem 16.3.2}
to the obvious morphism from the root datum of \(\PGL_2\) to
that of \(\bG_{\pm\alpha}\).)
A \(\mf{pgl}_2\)-calculation then establishes the result.
\end{proof}

\begin{cor}
\label{cor:duality}
With the notation of Lemma \ref{lem:opp-ord},
there is a non-canonical \(\pint_\alpha\)-module isomorphism of
\(\pmb\fg_\alpha(\field_\alpha)_{x, i:j}\)
with
\(\pmb\fg_{-\alpha}(\field_\alpha)_{x, \tilde\jmath:\tilde\imath}\).
\end{cor}

\begin{proof}
This follows from Lemma \ref{lem:opp-ord} and the fact that
a choice of \(\field_\alpha\)-basis of
\(\pmb\fg_\alpha(\field_\alpha)\) furnishes an isomorphism
of \(\pmb\fg_\alpha(\field_\alpha)_{x, i:j}\)
with a quotient of fractional ideals of \(\pint_\alpha\).
Such a quotient is self-dual (as an \(\pint_\alpha\)-module).
\end{proof}

\begin{cor}
\label{cor:opp-ord}
For any \(\alpha \in \Root\),
we have that
\(\ord_x(-\alpha) = -\ord_x(\alpha)\).
\end{cor}

\begin{cor}
\label{cor:ord-symm}
For any \(\alpha \in \Root\symm\),
we have that
\(\ord_x(\alpha) = \Z_\alpha\)
or
\(\ord_x(\alpha) = \Z_\alpha + \frac1 2 e_\alpha\inv\).
\end{cor}

\subsection{Subgroups associated to functions on root systems}
\label{sec:MP-mixed-depth}

Most of this section is devoted establishing analogues of
the results of
\xcite{adler-spice:good-expansions}*
	{\S\xref{sec:concave_gen}}
for subgroups of \fg attached to functions on root systems.
The proofs here are considerably easier, because the group
structure of \fg is much simpler than that of \(G\),
and because Galois descent behaves better for algebras
than groups,
so that we get to ignore many technical complications.
For example, we do not have to restrict our attention to
\emph{concave} functions
\xcite{adler-spice:good-expansions}*
	{Definition \xref{defn:concave}},
or intersect with a parahoric subalgebra as in
\xcite{adler-spice:good-expansions}*
	{Definition \xref{defn:vGvr}}.

\begin{defn}
\label{defn:vgvr}
If \map f{\Gamma\bslash\wtilde\Root}{\tR \cup \sset{+\infty}}
is any function (not necessarily concave),
then put
\[
\lsub\bT\fg_{x, f}
= \Bigl(
	\bigoplus_{\substack
		{\alpha \in \wtilde\Root \\
		f(\alpha) < +\infty}
	}
		\pmb\fg_\alpha(\field_\alpha)_{x, f(\alpha)}
\Bigr)^\Gamma,
\]
with the notation of Definition \ref{defn:ord-x}.
(See also \cite{adler-spice:explicit-chars}*
	{\S\xref{sec:buildings}, p.~1144}.)
\end{defn}

The group algebra \(\pint[T]\) of the torus \(T\)
acts naturally on \fg, on \(\pmb\fg_\alpha\),
and on each \(\lsub T\fg_{x, f}\)
(\via their natural \pint-module structures,
and the adjoint action of \(T\)).
We prove a few basic results about its action on groups of
the form \(\lsub\bT\fg_{x, f}\) that will be needed in the
proof of Lemma \ref{lem:gxf-card}.

\begin{lem}
\label{lem:gxf-by-orbit}
There is a canonical \(\pint[T]\)-module isomorphism
\[
\fg \cong \bigoplus_{\alpha \in \Gamma\bslash\wtilde\Root}
	\pmb\fg_\alpha(\field_\alpha)
\]
that restricts,
for any function
\map f{\Gamma\bslash\wtilde\Root}{\tR \cup \sset{+\infty}},
to an isomorphism
\[
\lsub\bT\fg_{x, f}
\cong \bigoplus_{\alpha \in \Gamma\bslash\wtilde\Root}
	\pmb\fg_\alpha(\field_\alpha)_{x, f(\alpha)}.
\]
\end{lem}

\begin{proof}
The inverse isomorphism sends
\((X_\alpha)_{\alpha \in \Gamma\bslash\wtilde\Root}\)
to
\(\sum_{\alpha \in \Gamma\bslash\wtilde\Root}
	\sum_{\sigma \in \Gamma/\Gamma_\alpha}
		\sigma X_\alpha\).
It is straightforward to check that its restriction has the
desired properties.
\end{proof}

\begin{cor}
\label{cor:gxf-inv}
Suppose that
\map{f, g}{\Gamma\bslash\wtilde\Root}{\tR \cup \sset{+\infty}},
with
\begin{itemize}
\item \(f(0) = g(0)\),
\item \(f(\alpha) \le g(\alpha)\)
	for all \(\alpha \in \wtilde\Root\),
and
\item \(g(\alpha) < +\infty\)
	whenever \(f(\alpha) < +\infty\).
\end{itemize}
Then there is a non-canonical \pint-module isomorphism
\(\lsub\bT\fg_{x, f:g}
	\cong \lsub\bT\fg_{x, \tilde g:\tilde f}\),
where
\[
\tilde f(\alpha) = \begin{cases}
\widetilde{f(-\alpha)}, & f(-\alpha) < +\infty \\
+\infty,                & \text{otherwise,}
\end{cases}
\]
and similarly for \(\tilde g\).
\end{cor}

\begin{proof}
By Corollary \ref{cor:duality},
there is a non-canonical \pint-module isomorphism
\[
\bigoplus_{\alpha \in \Gamma\bslash\wtilde\Root}
	\pmb\fg_\alpha(\field_\alpha)
		_{x, f(\alpha):g(\alpha)}
\cong
\bigoplus_{\alpha \in \Gamma\bslash\wtilde\Root}
	\pmb\fg_{-\alpha}(\field_\alpha)
		_{x, \tilde g(-\alpha):\tilde f(-\alpha)},
\]
which, by Lemma \ref{lem:gxf-by-orbit}, may be transferred
to the desired isomorphism.
\end{proof}

\begin{lem}
\label{lem:gxf-shift}
If \((r_\alpha)_{\alpha \in \Gamma\bslash\Root}\)
is a collection of rational numbers such that
\(r_\alpha \in \Z_\alpha\)
	for all \(\alpha \in \Root\),
then there is a
non-canonical \(\pint[T]\)-module automorphism of \fg
that restricts,
for any function
\map f{\Gamma\bslash\wtilde\Root}{\tR \cup \sset{+\infty}},
to an isomorphism
\[
\lsub\bT\fg_{x, f}
\cong \lsub\bT\fg_{x, f + r}
\]
(where
\((f + r)(0) \ldef f(0)\)
and
\((f + r)(\alpha) \ldef f(\alpha) + r_\alpha\)
for all \(\alpha \in \Root\)),
and that has determinant there of valuation
\(\displaystyle\sum_{\substack{
	\alpha \in \Root \\
	f(\alpha) < +\infty
}} r_\alpha\).
\end{lem}

\begin{proof}
By hypothesis, there is a collection
\((c_\alpha)_{\alpha \in \Gamma\bslash\Root}\)
such that \(c_\alpha \in \field_\alpha\)
and \(\ord(c_\alpha) = r_\alpha\)
for all \(\alpha \in \Root\).
By Lemma \ref{lem:gxf-by-orbit},
the automorphism of
\(\bigoplus_{\alpha \in \Gamma\bslash\wtilde\Root}
	\pmb\fg_\alpha(\field_\alpha)\)
that is the identity on the \(\alpha = 0\) summand,
and that multiplies the summand corresponding to
\(\alpha \in \Root\) by \(c_\alpha\),
may be transferred to the desired automorphism of \fg.
\end{proof}

Lemma \ref{lem:gxf-card} is crucial in the calculation
of Proposition \ref{prop:const}, where it accounts for the
appearance in the character formula Theorem \ref{thm:ratl}
of a blow-up controlled by the discriminant.

\begin{lem}
\label{lem:gxf-card}
Suppose that
\map{f, g}{\Gamma\bslash\wtilde\Root}{\R \cup \sset{+\infty}}
are functions such that
	\begin{itemize}
	\item \(f(\alpha) \le g(\alpha)\)
for all \(\alpha \in \wtilde\Root\),
	\item \(f(\alpha) < +\infty\)
whenever \(f(-\alpha) < +\infty\), and similarly for
\(g(\alpha)\),
	\item \(g(\alpha) < +\infty\)
whenever \(f(\alpha) < +\infty\),
	and
	\item
\(f(\alpha) + f(-\alpha), g(\alpha) + g(-\alpha)
\in \Z_\alpha \cup \sset{+\infty}\)
for all \(\alpha \in \Root\).
	\end{itemize}
Then
\begin{multline*}
\card{\lsub\bT\fg_{x, f:g}}\dotm
\card{\lsub\bT\fg_{x, {f+}:{g+}}} \\
= \card{\ft_{f(0):g(0)}}\card{\ft_{{f(0)+}:{g(0)+}}}\dotm
\prod_{\substack{\alpha \in \Root \\ f(\alpha) < +\infty}}
	q^
		{(g(\alpha) + g(-\alpha))
			- (f(\alpha) + f(-\alpha))}.
\end{multline*}
\end{lem}

\begin{rem}
\label{rem:gxf-card}
In our application of Lemma \ref{lem:gxf-card}
(see the proof of Proposition \ref{prop:const}),
we have that
\(f(0) = g(0)\)
and that
\(f\) and \(g\) are \(\pm\Gamma\)-,
not just \(\Gamma\)-, invariant;
so then the condition on their values becomes that
\(f(\alpha), g(\alpha) \in \tfrac1 2\Z_\alpha\)
for all \(\alpha \in \Root\),
and the cardinality becomes
\[
\prod_{\substack{
	\alpha \in \langle\pm1\rangle\bslash\Root \\
	f(\alpha) < +\infty
}} q^{2(g(\alpha) - f(\alpha))}.
\]
\end{rem}

\begin{proof}
We have that
\(\lsub\bT\fg_{x, f:g}
= \ft_{x, f(0):g(0)}
\oplus \fg_{x, f_\Root:g_\Root}\),
where \(f_\Root\) and \(g_\Root\) agree with \(f\) and \(g\),
respectively, on \Root,
and take the value \(g(0)\) at \(0\);
and similarly for \(\lsub\bT\fg_{x, {f+}:{g+}}\).
Thus, it suffices to consider the case where \(f(0) = g(0)\).
In this case, put \(r_\alpha = 0\)
if \(f(\alpha) = +\infty\),
and
\[
r_\alpha
= g(\alpha) - (\widetilde{g+})(\alpha)
= g(\alpha) + g(-\alpha) \in \Z_\alpha
\]
otherwise.
By Corollary \ref{cor:gxf-inv}
and Lemma \ref{lem:gxf-shift},
we have that
\[
\lsub\bT\fg_{x, {f+}:{g+}}
\cong \lsub\bT\fg_{x, \widetilde{g+}:\widetilde{f+}}
\cong \lsub\bT\fg
	_{x, ((\widetilde{g+}) + r):((\widetilde{f+}) + r)}
= \lsub\bT\fg
	_{x, g:((\widetilde{f+}) + r)},
\]
hence that
\[
\card{\lsub\bT\fg_{x, f:g}}
	\dotm\card{\lsub\bT\fg_{x, {f+}:{g+}}}
= \card{\lsub\bT\fg_{x, f:g}}
	\dotm\smcard{\lsub\bT\fg_{x, g:((\widetilde{f+}) + r)}}
= \smcard{\lsub\bT\fg_{x, f:((\widetilde{f+}) + r)}}.
\]
Now, whenever \(f(\alpha) < +\infty\), we have that
\[
s_\alpha
\ldef ((\widetilde{f+}) + r)(\alpha) - f(\alpha)
= (g(\alpha) + g(-\alpha)) - (f(\alpha) + f(-\alpha))
\]
lies in \(\Z_\alpha\),
so, by Lemma \ref{lem:gxf-shift} again,
\[
\smcard{\lsub\bT\fg_{x, f:((\widetilde{f+}) + r)}}
= \bigabs{\det\bigl(
	\anonmap{\lsub\bT\fg_{x, f}}
	{\lsub\bT\fg_{x, ((\widetilde{f+}) + r)}}
\bigr)}\inv
= \prod_{\alpha \in \Root}
	q^{s_\alpha}.\qedhere
\]
\end{proof}

\numberwithin{thm}{subsection}

\section{Character formul\ae}
\label{sec:ratl}

As in \S\ref{sec:MP}, let \bG be a connected, reductive
\field-group.
This section is devoted to making more explicit the
computations of \cite{adler-spice:explicit-chars}
of the characters of supercuspidal representations of \(G\).
To use those computations, we must assume that \bG
satisfies \xcite{adler-spice:good-expansions}*
	{Hypotheses \xref{hyp:conn-cent}
and \xref{hyp:good-weight-lattice}},
so we do so for the remainder of the paper.
In order that the theory not be vacuous, we also assume for
the remainder of the paper that
\bG possesses a \tamefield-split, maximal torus,
so that \bG also satisfies
\xcite{adler-spice:good-expansions}*
	{Hypotheses \xref{hyp:reduced}
and \xref{hyp:torus-H1-triv}}
(see \xcite{adler-spice:good-expansions}*
	{Remark \xref{rem:when-hyps-hold}}).

We begin by recalling Yu's construction \cite{yu:supercuspidal}
of the so called `tame' supercuspidal representations of \(G\);
the necessary notation is introduced in \S\ref{sec:char-notn}.
The next four subsections handle the ingredients of the
character formul{\ae} of \cite{adler-spice:explicit-chars},
as described in the introduction.
Finally, in \S\ref{sec:char-char}, we tie together these
results to give an inductive formula (Theorem \ref{thm:ratl})
that computes the character of a tame supercuspidal
representation in terms of the character of a supercuspidal
representation of a smaller group,
together with what we believe to be
a minimal collection of auxiliary data
(the fourth roots of unity occurring in
Definition \ref{defn:signs}).

\subsection{Notation}
\label{sec:char-notn}

We begin by defining,
inductively (on the semisimple rank of \bG, say),
a parametrising set for the supercuspidal
representations that we consider.

\begin{defn}
\label{defn:cusp-pair}
A \term[cuspidal pair]{cuspidal \(G\)-pair} is a pair
\((\bG', \pi')\), where
	\begin{itemize}
	\item \(\bG'\) is a (\field-)subgroup of \bG such that
\(Z(\bG')/Z(\bG)\) is (\field-)anisotropic,
	\item \(\bG'_\tamefield\) is a Levi subgroup of
\(\bG_\tamefield\),
	\item \(\pi'\) is a tame supercuspidal representation
of \(G'\),
	and
	\item there is a character \(\phi\) of \(G'\)
such that the depth \cite{moy-prasad:k-types}*{Theorem 5.2}
of \(\pi' \otimes \phi\inv\) is strictly less than that of
\(\pi'\).
	\end{itemize}
The \term{depth} of the cuspidal pair is the
depth of \(\pi'\).
(Recall that a \term{tame} supercuspidal representation is one that
arises \via Yu's construction \cite{yu:supercuspidal}*{\S4}
from a datum involving generic characters.
See Definition \ref{defn:cusp-ind}.)

The pair is
\term[cuspidal pair!toral]{toral}
if \(\bG'\) is a torus,
and
\term[cuspidal pair!compact]{compact}
if \(\bG'/Z(\bG)\) is \field-anisotropic.
We sometimes write \((G', \pi')\) instead of
\((\bG', \pi')\).
(Since \(G'\) is Zariski-dense in \(\bG'\)
\cite{borel:linear}*{Corollary 18.3},
this notational convenience causes no ambiguity.)
\end{defn}


The somewhat artificial condition on depth in
Definition \ref{defn:cusp-pair} ensures that
the character \(\phi_{d - 1}\) of
Definition \ref{defn:cusp-ind} is not trivial,
hence that the `induction' \(\YuUp_{G'}^G \pi'\) defined
there makes sense.

\begin{defn}
\label{defn:cusp-ind}
By the definition of a tame supercuspidal representation,
if \((\bG', \pi')\) is a depth-\(r\), cuspidal \(G\)-pair,
then \(\pi'\) is associated to a (non-unique) generic
\(G'\)-datum
\[
\bigl(
	(\bG^0 \subsetneq \dotsb \subsetneq \bG^{d - 1} = \bG'),
	x,
	(r_0, \dotsc, r_{d - 1}),
	(\phi_0, \dotsc, \phi_{d - 1}),
	\rho
\bigr),
\]
in the sense of \cite{jkim:exhaustion}*{\S0.3}
(see also \cite{yu:supercuspidal}*{\S3}).
In particular, \(r_{d - 1} = r\);
and, by Definition \ref{defn:cusp-pair}, \(\phi_{d - 1} \ne 1\).
Then \((G', \pi', \phi_{d - 1}, x, X^*)\) is called an
\term
	[cuspidal pair!corresponding expanded quintuple]
	{expanded cuspidal \(G\)-quintuple}
corresponding to \((G', \pi')\) exactly when
\(X^*\) is (\((\bG', \bG)\)-)generic,
in the sense of \cite{yu:supercuspidal}*{\S8, p.~596},
and realises \(\phi_{d - 1}\) on \(G'_{x, r}\)
(relative to \(x\)),
in the sense of \cite{yu:supercuspidal}*{\S5, p.~593}.
By definition, every cuspidal \(G\)-pair has a corresponding
expanded cuspidal \(G\)-quintuple.

Write \mterm{\YuUp_{G'}^G \pi'} for the
supercuspidal representation of \(G\) constructed by
Yu \cite{yu:supercuspidal}*{\S4} from the generic \(G\)-datum
\[
\bigl(
	(\bG^0 \subsetneq \dotsb \subsetneq
		\bG^{d - 1} = \bG' \subsetneq \bG^d = \bG),
	x,
	(r_0, \dotsc, r_{d - 1}, r_{d - 1}),
	(\phi_0, \dotsc, \phi_{d - 1}, 1),
	\rho
\bigr).
\]
By \cite{hakim-murnaghan:distinguished-tame-sc}*{Theorem 6.6},
under a centrality condition analogous to our
Definition \ref{defn:cent-cusp-pair} below
\cite{hakim-murnaghan:distinguished-tame-sc}*
	{Hypothesis C(\bG) and Remark 2.49},
the representation \(\YuUp_{G'}^G \pi'\) depends only on \(\pi'\),
not on the choice of generic \(G\)-datum;
but this deep fact is used here only as a notational
convenience.
\end{defn}

\begin{rem}
\label{rem:uniq-expand}
If \((G', \pi', \phi, x, X^*)\)
and \((G', \pi', \phi', x', X^{\prime\,*})\) are
depth-\(r\), expanded cuspidal \(G\)-quintuples
corresponding to the same cuspidal \(G\)-pair \((G', \pi')\),
then, by \cite{yu:supercuspidal}*{Lemma 3.3},
the images of \(x\) and \(x'\) in the reduced building
\(\rBB(\bG, \field)\) are in the same \(G'\)-orbit.
By construction \cite{yu:supercuspidal}*{\S4},
since \(\pi'\) is the twist by \(\phi\) of a representation
of depth strictly less than \(r\),
we have that \(\res\pi'to{G'_{x, r}}\) is \(\phi\)-isotypic,
and similarly for \(\phi'\).
Thus, \(\phi\) and \(\phi'\) agree on
(the \(G'\)-orbit of) \(G'_{x, r}\),
so that
\(X^{\prime\,*} \in X^* + \mf z(\fg')^*_{{(-r)+}}\).
We `rigidify' \(X^*\) further in
Definition \ref{defn:cent-cusp-pair} below.
\end{rem}

The next definition builds in
\xcite{adler-spice:explicit-chars}*
	{Hypothesis \xref{hyp:X*-central}}
to our definition of a cuspidal pair.

\begin{defn}
\label{defn:cent-cusp-pair}
In the notation of Definition \ref{defn:cusp-ind},
the depth-\(r\), expanded cuspidal \(G\)-quintuple
\((G', \pi', \phi, x, X^*)\) is
\term
	[cuspidal pair!corresponding extended \(G\)-quintuple!central]
	{central}
exactly when \(X^*\) realises \(\phi\) on \(G'_{x, {(r/2)+}}\)
(not just on \(G'_{x, r}\)),
in the sense of \cite{yu:supercuspidal}*{\S5, p.~593};
and a cuspidal \(G\)-pair is
\term[cuspidal pair!central]{central}
if it has a corresponding central, expanded cuspidal
\(G\)-quintuple.
\end{defn}

Note that a toral, cuspidal \(G\)-pair
is automatically compact and central.

\begin{rem}
\label{rem:uniq-cent-expand}
In the notation of Remark \ref{rem:uniq-expand},
if the expanded cuspidal quintuples occurring there are
central, then
\(X^{\prime\,*} \in X^* + \mf z(\fg')^*_{x, -r/2}\).
\end{rem}

We now introduce notation for the basic objects of
harmonic analysis, namely, characters and
(Fourier transforms of) orbital integrals.
We work with characters and orbital integrals
that have been `normalised' by multiplying
by the square root of the (reduced) discriminant
to remove singularities.

\begin{defn}
\label{defn:normal-harm}
Recall that, if \(\pi\) is a smooth irreducible
representation of \(G\)
and \(X^*\) is a semisimple element of \(\fg^*\)
(i.e., if it is fixed by the co-adjoint action of some
maximal torus of \(G\)),
then the distribution character
\[
\anonmapto F{\tr \pi(F)}
\]
on \(G\) and the Fourier transform
(see Definition \ref{defn:FT}) of the orbital integral
\[
\anonmapto f
	{\int_{G/C_G(X^*)\conn} f(\Ad^*(g)X^*)\textup d\dot g}
\]
on \(\fg^*\) are represented, at least on
\(G\rss\) and \(\fg\rss\), respectively,
by scalar functions \(\Theta_\pi\)
(\cite{hc:submersion}*{p.~99, Corollary}
and \cite{adler-debacker:mk-theory}*{Theorem B.1.1})
and \(\hat\mu^G_{X^*}\)
\cite{adler-debacker:mk-theory}*{Theorem A.1.2},
respectively.
The function \(\Theta_\pi\) is canonical,
but normalising the function \(\hat\mu^G_{X^*}\) requires
making a choice of \(G\)-invariant measure on
\(G/C_G(X^*)\conn\).
We use the quotient of Waldspurger's canonical Haar measures
\cite{waldspurger:nilpotent}*{\S I.4} on \(G\) and its
(reductive \cite{borel:linear}*{Proposition 13.19})
subgroup \(C_G(X^*)\conn\).
(These measures assign mass
\(\card{\ol K}\dotm\card{\Lie(\ol K)}^{-1/2}\)
to a parahoric subgroup \(K\) with reductive quotient \ol K;
see \cite{debacker-reeder:depth-zero-sc}*{\S5.1, p.~835}.)
Put
\[
\mterm{\Phi_\pi}
= \smabs{\redD_G}^{1/2}\Theta_\pi
\qandq
\mterm{\hat O^G_{X^*}}
= \smabs{\redD_G(X^*)}^{1/2}\smabs{\redD_G}^{1/2}
	\hat\mu^G_{X^*},
\]
and, if \(X^*\) is strongly regular semisimple,
then also
\[
\mterm{\smash{\widehat{SO}}^G_{X^*}}
= \sum_
	{X^{\prime\,*} \in G\bslash\bG\dota X^*}
	\hat O^G_{X^{\prime\,*}},
\]
where
\(\bG\dota X^* \ldef \Ad^*(\bG(\unfield))X^* \cap \fg^*\)
is the \bG-stable conjugacy class of \(X^*\)
(see \cite{kottwitz:ratl-conj}*{\S3, p.~788}
and
\cite{debacker-reeder:depth-zero-sc}*{\S2.9, p.~817}),
as functions defined at least on
\(G\rss\), \(\fg\rss\), and \(\fg\rss\), respectively.
\end{defn}

\begin{rem}
\label{rem:extend-Theta}
The functions \(\Theta_\pi\), \(\hat O^G_{X^*}\), and
\(\smash{\widehat{SO}}^G_{X^*}\)
may extend smoothly off the regular semisimple set; for
example, they are globally defined if \(G/Z(G)\) is compact.
Since \(G\rss\) is dense in \(G\), any such
extension is uniquely determined.
Therefore, we do not distinguish notationally
between the functions and their maximal smooth extensions.
\end{rem}


For the remainder of this section,
fix a compact, central, expanded cuspidal quintuple
\index{notation}{\(\bG'\)}
\index{notation}{\(\pi'\)}
\((G', \pi', \phi, x, X^*)\).
Let \(r > 0\) be the depth of the quintuple.
Fix also an \(r\)-approximable element
\(\mterm\gamma \in G\),
say with \(r\)-approximation
\(\gamma = \mterm{\gamma_{< r}}\mterm{\gamma_{\ge r}}\)
\xcite{adler-spice:good-expansions}*
	{Definition \xref{defn:r-approx}}.
Until \S\ref{sec:index}, put \(\mterm\bH = \CC\bG r(\gamma)\)
\xcite{adler-spice:good-expansions}*
	{Definition \xref{defn:fancy-centralizer-no-underline}}.
(In \S\S\ref{sec:index} and \ref{sec:orbital}, the connected
group \(\CC\bG r(\gamma)\) is enlarged to the possibly
disconnected group \(C_\bG(\gamma_{< r})\),
of which \(\CC\bG r(\gamma)\) is the identity component
\xcite{adler-spice:good-expansions}*
	{Corollary \xref{cor:compare-centralizers}}.)

Let \bT be a maximal \(H\)-torus
(so that \(\gamma_{< r} \in T\)),
and write
\begin{itemize}
\item \(\mterm\Root = \Root(\bG, \bT)\),
\item \(\mterm{\wtilde\Root} = \wtilde\Root(\bG, \bT)\),
\item \(\mterm{\Root_\bH} = \Root(\bH, \bT)\),
and
\item \(\mterm{\wtilde\Root_\bH} = \wtilde\Root(\bH, \bT)\).
\end{itemize}
It is convenient to impose no additional hypotheses on \bT
for Lemma \ref{lem:part-disc}, but we do add some hypotheses
immediately afterwards.

\begin{rem}
\label{rem:compare-centralizers}
By \xcite{adler-spice:good-expansions}*
	{Lemma \xref{lem:compare-centralizers}},
if \(\gamma \in T\), then \(\gamma_{< r} \in T\) as well.
In this case, we have that
\(\ord\bigl(
	\alpha(\gamma) - \alpha(\gamma_{< r})
\bigr) \ge r\)
(hence that
\(\ord_\gamma \alpha = \ord_{\gamma_{< r}} \alpha\))
for \(\alpha \in \wtilde\Root \setminus \wtilde\Root_\bH\),
and that
\(\alpha(\gamma) = \alpha(\gamma_{\ge r})\) for
\(\alpha \in \wtilde\Root_\bH\).
In particular,
\(\wtilde\Root_\bH = \set
	{\alpha \in \smash{\wtilde\Root}}
	{\ord_\gamma \alpha \ge r}\).
\end{rem}

\begin{lem}
\label{lem:part-disc}
If \(\gamma \in T\), then
\(\abs{\redD_G(\gamma)}
= \abs{\redD_G(\gamma_{< r})}\dotm
	\abs{\redD_H(\gamma_{\ge r})}\).
\end{lem}

\begin{proof}
By Remark \ref{rem:disc:roots}, we have that
\begin{multline*}
\abs{\redD_G(\gamma)}
= \prod_{\substack
	{\alpha \in \Root \\
	\alpha(\gamma) \ne 1}
}
	q^{-\ord_\gamma \alpha}\text,\quad
\abs{\redD_G(\gamma_{< r})}
= \prod_{\substack
	{\alpha \in \Root \\
	\alpha(\gamma_{< r}) \ne 1}
}
	q^{-\ord_{\gamma_{< r}} \alpha}\text, \\
\andq
\abs{\redD_H(\gamma_{\ge r})}
= \prod_{\substack
	{\alpha \in \Root_\bH \\
	\alpha(\gamma_{\ge r}) \ne 1}
}
	q^{-\ord_{\gamma_{\ge r}} \alpha},
\end{multline*}
so that the desired equality follows from
Remark \ref{rem:compare-centralizers}.
\end{proof}

For the remainder of this section, suppose that
\bT is \tamefield-split and contained in \(\bG'\),
and that \(x \in \BB(\bT, \field)\).
Then \(\gamma_{< r} \in T \subseteq G'\)
\xcite{adler-spice:good-expansions}*
	{Remark \xref{rem:approx-facts-in-center}},
so that it makes sense to define
\(\CC{\bG'}r(\gamma_{< r})\).
By abuse of notation, we write
\mterm{\CC{\bG'}r(\gamma)}, or just \mterm{\bH'},
in place of \(\CC{\bG'}r(\gamma_{< r})\).
Put \(\mterm{\Root'} = \Root(\bG', \bT)\)
and
\(\mterm{\Root_{\bH'}} = \Root(\bH', \bT)
	= \Root_\bH \cap \Root'\).

\begin{rem}
\label{rem:funny-centralizer-descends}
By \xcite{adler-spice:good-expansions}*
	{Lemma \xref{lem:funny-centralizer-descends}},
we have that \(\bH' = \bH \cap \bG'\).
\end{rem}

In order to minimise notation, it helps to have a
notation for the Fourier transforms of orbital integrals
occurring in character formul{\ae} that does not require us
to refer directly to the element \(X^*\).

\begin{defn}
\label{defn:O-pi}
Write
\(\hat O^H_{\pi'}\)
	and (if \(X^*\) is strongly regular semisimple,
that is, if \(\bG'\) is a torus)
\(\smash{\widehat{SO}}^H_{\pi'}\)
for the restrictions to \(\fh\rss \cap \fh_r\) of
\(\hat O^H_{X^*}\)
	and \(\smash{\widehat{SO}}^H_{X^*}\),
respectively.
\end{defn}

Since \(\res\pi'to{G_{x, r}}\) is
\(\res\phi to{G_{x, r}}\)-isotypic,
we have by \xcite{adler-spice:explicit-chars}*
	{Lemma \xref{lem:orbital-cancel}}
that \(\hat O^H_{\pi'}\) and \(\smash{\widehat{SO}}^H_{\pi'}\)
depend only on \(\pi'\)
(indeed, only on \(\res\phi to{G_{x, r}}\)),
not on the choice of expanded quintuple.

\subsection{Indices}
\label{sec:indices}

The character formul{\ae} of
\cite{adler-spice:explicit-chars} involve
groups \dc and \(\odc{\gamma_{< r}; x, {r+}}\).
These are defined in
\xcite{adler-spice:good-expansions}*
	{Definition \xref{defn:bracket}},
but, since we are interested only in the cardinalities of
certain quotients involving these groups, we do not
reproduce their full definitions here
(though see the proof of Proposition \ref{prop:const}
below).
Since the definitions of these groups depend only on
\(\gamma_{< r}\), we ease notation by assuming
throughout this section that \(\gamma = \gamma_{< r}\).

\begin{prop}
\label{prop:const}
\begin{multline*}
\bigindx\dc{T_{0+}G_{x, r/2}}\dotm
\bigindx\dcp{T_{0+}G_{x, {(r/2)+}}} \\
= \card{\ft_{0:{0+}}}\dotm\card{\fh_{x, 0:{0+}}}\inv\dotm
\prod_{\alpha \in \Root_\bH}
	q^r
\dotm
\smabs{\redD_G(\gamma)}\inv.
\end{multline*}
\end{prop}

\begin{proof}
In the notation of
\xcite{adler-spice:good-expansions}*
	{Definition \xref{defn:vGvr}},
write
\(\lsub\bT G_{x, {0+} \vee (r - \ord_\gamma)/2}\)
and \(\lsub\bT\fg_{x, {0+} \vee (r - \ord_\gamma)/2}\)
as abbreviations for \(\lsub\bT G_{x, f_\bH}\)
and \(\lsub\bT\fg_{x, f_\bH}\),
where
\[
f_\bH(\alpha) = \begin{cases}
{0+}, &
	\alpha \in \wtilde\Root_\bH \\
\tfrac1 2(r - \ord_\gamma \alpha), &
	\alpha \in \wtilde\Root \setminus \wtilde\Root_\bH;
\end{cases}
\]
and, similarly,
\(\lsub\bT G_{x, (0 \vee {(r - \ord_\gamma)/2)+}}\)
and \(\lsub\bT\fg_{x, (0 \vee {(r - \ord_\gamma)/2)+}}\)
instead of \(\lsub\bT G_{x, {f+}}\)
and \(\lsub\bT\fg_{x, {f+}}\),
where \(f+\) is defined by
\[
({f+})(\alpha) = {f_\bH(\alpha)+} = \begin{cases}
{0+}, &
	\alpha \in \wtilde\Root_\bH \\
{\tfrac1 2(r - \ord_\gamma \alpha)+}, &
	\alpha \in \wtilde\Root \setminus \wtilde\Root_\bH.
\end{cases}
\]
Note that \(f_\bH\) (hence also \(f+\))
is everywhere strictly positive
.
We are also interested in the merely
non-negative function \(f\) defined by
\[
f(\alpha) = \begin{cases}
0, & \alpha \in \wtilde\Root_\bH \\
\tfrac1 2(r - \ord_\gamma \alpha), &
	\alpha \in \wtilde\Root \setminus \wtilde\Root_\bH,
\end{cases}
\]
and, as above, we write
\(\lsub\bT\fg_{x, 0 \vee (r - \ord_\gamma)/2}\)
instead of
\(\lsub\bT\fg_{x, f}\).
(The corresponding group is not of interest to us.)

By \xcite{adler-spice:good-expansions}*
	{Definition \xref{defn:bracket}},
\[
\dc = \lsub\bT G_{x, {0+} \vee (r - \ord_\gamma)/2}
\qandq
\dcp = \lsub\bT G_{x, 0 \vee {(r - \ord_\gamma)/2)+}}.
\]
In the notation of
\xcite{adler-spice:good-expansions}*
	{Definition \xref{defn:vGvr}},
we have by
\xcite{adler-spice:good-expansions}*
	{Proposition \xref{prop:heres-a-gp}}
that
\(T_{0+}G_{x, r/2} = (T, G)_{x, ({0+}, r/2)}\)
and \(T_{0+}G_{x, {(r/2)+}} = (T, G)_{x, ({0+}, {(r/2)+})}\);
the corresponding Lie-algebra statements follow immediately
from Definition \ref{defn:vgvr}.
Thus, we are almost in a position to use
Lemma \ref{lem:gxf-card}, except for two obstacles.
First, we are working on the group, not on the Lie algebra;
and, second, the function \(f_\bH\) is not real valued.

The solution proceeds in two stages:
we first pass to the Lie algebra,
then `manually' account for the difference between \(f_\bH\)
and the real-valued function \(f\).
Although there need not be any reasonably behaved map from
\(\lsub\bT G_{x, ({0+} \vee (r - \ord_\gamma)/2)}
	/T_{0+}G_{x, r/2}\)
to
\(\lsub\bT\fg_{x, ({0+} \vee (r - \ord_\gamma)/2)}
	/(\ft_{{0+}} + \fg_{x, r/2})\),
an argument as in \xcite{adler-spice:good-expansions}*
	{Corollary \xref{cor:gen-iwahori-factorization-bij}}
shows that there are filtrations
\((\mc G_i)_{i = 0}^n\) and \((\Lie(\mc G_i))_{i = 0}^n\)
of \(\lsub\bT G_{x, {0+} \vee (r - \ord_\gamma)/2}\)
and \(\lsub\bT\fg_{x, {0+} \vee (r - \ord_\gamma)/2}\),
respectively,
such that
	\begin{itemize}
	\item
\(\mc G_i/\mc G_{i + 1}
\cong \Lie(\mc G_i)/\Lie(\mc G_{i + 1})\)
for all \(i \in \sset{0, \dotsc, n - 1}\)
(\via a Moy--Prasad isomorphism;
see \cite{moy-prasad:k-types}*{\S3.7}
and \cite{adler:thesis}*{\S1.5, p.~12}),
	\item
\(\mc G_0 = \lsub\bT G_{x, {0+} \vee (r - \ord_\gamma)/2}\)
and
\(\Lie(\mc G_0) = \lsub\bT\fg_{x, {0+} \vee (r - \ord_\gamma)/2}\),
	and
	\item
\(\mc G_n = T_{0+}G_{x, r/2}\)
and
\(\Lie(\mc G_n) = \ft_{0+} + \fg_{x, r/2}\).
	\end{itemize}
Thus,
\[
\bigindx\dc{T_{0+}G_{x, r/2}}
= \bigindx
	{\lsub\bT\fg_{x, ({0+} \vee (r - \ord_\gamma)/2)}}
	{\ft_{0+} + \fg_{x, r/2}}
\]
and, similarly,
\[
\bigindx\dcp{T_{0+}G_{x, {(r/2)+}}}
= \bigindx
	{\lsub\bT\fg_{x, (0 \vee {(r - \ord_\gamma)/2)+}}}
	{\ft_{0+} + \fg_{x, {(r/2)+}}},
\]
so that we may re-write the left-hand side of the equality in
the statement as
\begin{multline*}
\underbrace
	{\indx{\ft_0 + \fg_{x, r/2}}{\ft_{0+} + \fg_{x, r/2}}\dotm
	\card{\lsub\bT\fg_
		{x, (0 \vee (r - \ord_\gamma)/2):%
			({0+} \vee (r - \ord_\gamma)/2)}
	}\inv}_{(*)}\times{} \\
\underbrace{
	\bigindx
		{\lsub\bT\fg_{x, (0 \vee (r - \ord_\gamma)/2)}}
		{\ft_0 + \fg_{x, r/2}}
	\dotm\bigindx
		{\lsub\bT\fg_{x, (0 \vee {(r - \ord_\gamma)/2)+}}}
		{\ft_{0+} + \fg_{x, {(r/2)+}}}
}_{(**)}.
\end{multline*}
Since \(\tfrac1 2(r - \ord_\gamma \alpha) \le 0\)
exactly when \(\ord_\gamma \alpha \ge r\),
\ie, exactly when \(\alpha \in \wtilde\Root_\bH\)
(see Remark \ref{rem:compare-centralizers}),
we have
\[
(*) = \card{\ft_{0:{0+}}}\dotm\card{\fh_{x, 0:{0+}}}\inv.
\]
By Lemma \ref{lem:gxf-card} (and Remark \ref{rem:gxf-card})
and
\xcite{adler-spice:good-expansions}*
	{Corollary \xref{cor:compare-centralizers}}, we have
\begin{align*}
(**) ={} &
\prod_{\alpha \in \Root_\bH} q^{2(r/2 - 0)}
\dotm
\prod_{\alpha \in \Root \setminus \Root_\bH}
	q^{2[r/2 - (r - \ord_\gamma \alpha)/2]} \\
={} &
\prod_{\alpha \in \Root_\bH} q^r
\dotm
\prod_{\alpha \in \Root \setminus \Root_\bH}
	q^{\ord_\gamma \alpha} \\
={} &
\prod_{\alpha \in \Root_\bH} q^r
\dotm
\abs{\redD_G(\gamma)}.\qedhere
\end{align*}
\end{proof}

\begin{cor}
\label{cor:const}
\begin{multline*}
\bigindx{\odc{\gamma; x, r}}{
	\odc{\gamma; x, r}_{G'}G_{x, r/2}
}\dotm\bigindx{\odc{\gamma; x, {r+}}}{
	\odc{\gamma; x, {r+}}_{G'}G_{x, {(r/2)+}}
} \\
= \card{\fh_{x, 0:{0+}}}\inv\card{\fh'_{x, 0:{0+}}}\dotm
\abs{\redD_H(X^*)}\dotm
\abs{\redD_G(\gamma)}\inv\abs{\redD_{G'}(\gamma)}.
\end{multline*}
\end{cor}

\begin{proof}
By
\xcite{adler-spice:good-expansions}*
	{Proposition \xref{prop:heres-a-gp}
and Lemma \xref{lem:more-vGvr-facts}},
we have that
\(\odc{\gamma; x, d}_{G'} \cap T_{0+}G_{x, d/2}
= T_{0+}G'_{x, d/2}\)
for \(d \in \sset{r, {r+}}\); so the left-hand side of the
equality in the statement is
\begin{multline*}
\bigl(
	\bigindx{\odc{\gamma; x, r}}{T_{0+}G_{x, r/2}}\dotm
	\bigindx{\odc{\gamma; x, {r+}}}{T_{0+}G_{x, {(r/2)+}}}
\bigr)\times{} \\
\bigl(
	\bigindx{\odc{\gamma; x, r}_{G'}}{T_{0+}G'_{x, r/2}}\dotm
	\bigindx{\odc{\gamma; x, {r+}}_{G'}}{T_{0+}G'_{x, {(r/2)+}}}
\bigr)\inv.
\end{multline*}
By Proposition \ref{prop:const}, this equals
\[
\card{\fh_{x, 0:{0+}}}\inv\card{\fh'_{x, 0:{0+}}}\dotm
\underbrace{
	\prod_{\alpha \in \Root_\bH \setminus \Root_{\bH'}}
		q^r
}_{(*)}\dotm
\smabs{\redD_G(\gamma)}\inv
\smabs{\redD_{G'}(\gamma)}.
\]
By Remark \ref{rem:funny-centralizer-descends},
we have that
\(\bH' = \bH \cap \bG' = C_\bH(X^*)\).
Since \(\Root_{\bH'} = \Root_\bH \cap \Root'\),
we have by \cite{yu:supercuspidal}*{\S8, \textbf{GE1}}
that \(\ord_{X^*} \alpha = -r\)
 for the roots \(\alpha\) in the product ($*$); so
\[
(*)
= \prod_{\alpha \in \Root_\bH \setminus \Root_{\bH'}}
	q^{-\ord_{X^*} \alpha}
= \abs{\redD_H(X^*)}.\qedhere
\]
\end{proof}

\subsection{Roots of unity}
\label{sec:root}

The character formul{\ae} of
\cite{adler-spice:explicit-chars} involve \(4\)th roots of
unity defined in terms of Galois actions on root systems.
Definition \ref{defn:signs} below breaks down these roots in
a slightly different way from what is done in
\cite{adler-spice:explicit-chars};
in Proposition \ref{prop:root}, we show that the two
approaches give the same result.

Except for the symbols whose definitions make explicit
reference to the group \(\bH = \CC\bG r(\gamma)\),
the definitions below do not depend on the approximability
of \(\gamma\).

\begin{defn}
\label{defn:signs}
Write
\begin{itemize}
\item \(\mterm{\Root_{x, r/2}}
= \set{\alpha \in \Root}{r/2 \in \ord_x(\alpha)}\),
\item \(\mterm{\Root_{x, (r - \ord_\gamma)/2}}
= \sett{\alpha \in \Root}
	{\(\alpha(\gamma) \ne 1\)
	and \((r - \ord_\gamma \alpha)/2 \in \ord_x(\alpha)\)}\),
and
\item \(\mterm{\Root(\pi', \gamma)}
= \Root_{x, (r - \ord_\gamma)/2} \setminus \Root'\).
\end{itemize}
As in \xcite{adler-spice:explicit-chars}*
	{Proposition \xref{prop:gauss-sum}},
write also
	\begin{itemize}
	\item \(\mf G = q^{-1/2}\sum_{t \in \ff} \AddChar(t^2)\);
	\item
\(t_\alpha
= \tfrac1 2 e_\alpha\dotm
\Norm_{\field_\alpha/\field_{\pm\alpha}}(w_\alpha)\dotm
\textup d\alpha^\vee(X^*)\dotm
(\alpha(\gamma) - 1)\),
where
\(w_\alpha\) is an element of \(\field_\alpha\)
of valuation
\(\tfrac1 2(r - \ord_\gamma \alpha)\)
such that \(w_\alpha^2 \in \field_{\pm\alpha}\),
for \(\alpha \in \Root(\pi', \gamma)\);
and
	\item \(\bG_{\pm\alpha}\) for the
\(\field_{\pm\alpha}\)-group generated by the
(\(\pm\alpha\))-root subgroups of \bG.
	\end{itemize}

Now we introduce a number of new pieces of notation, each of
which builds on the others.
If \(A\) is a finite, Abelian group of even
order, then denote by \mterm{\sgn_A} the unique
non-trivial homomorphism from \(A\) to \(\sset{\pm1}\).
First put
{\def\prodder{%
\prod_
	{\alpha
		\in \Gamma\bslash\Root(\pi', \gamma)\symmram}%
}
\begin{align*}
\mterm{\tilde e(G, T, \gamma)}
={} & (-1)^{\smcard{\Gamma\bslash\Root_{x, (r - \ord_\gamma)/2}}}, \\
\mterm{\varepsilon\symmunram(G, T, \gamma)}
={} & \prod_{\alpha
		\in \Gamma\bslash\Root\textsub{symm, unram, \(x, r/2\)}}
	\sgn_{\ff_\alpha^1}(\alpha(\gamma)), \\
\mterm{\varepsilon\nosymm_{x, r/2}(G, T, \gamma)}
={} & \prod_{\alpha \in {\pm\Gamma}\bslash\Root\nosymm_{x, r/2}}
	\sgn_{\ff_\alpha}(\alpha(\gamma)), \\
\intertext{and}
\mterm{\varepsilon\noram_{x, r/2}(G, T, \gamma)}
={} & \varepsilon\nosymm_{x, r/2}(G, T, \gamma)
	\varepsilon\symmunram(G, T, \gamma). \\
\intertext{Next put}
\mterm{\tilde e(G/G', \gamma)}
={} & \frac
	{\tilde e(G, T, \gamma)}
	{\tilde e(G', T, \gamma)}, \\
\intertext{and similarly for
{\def\amterm#1{%
	\index{notation}{\(#1(G/G', \gamma)\)}%
	\me{#1}%
}
\amterm{\varepsilon\symmunram},
\amterm{\varepsilon\nosymm_{x, r/2}},
and \amterm{\varepsilon\noram_{x, r/2}}};
and then}
\mterm{e(G/G', \gamma)}
={} & \prod_{\substack{
	\alpha \in \Root \setminus \Root' \\
	\alpha(\gamma) \ne 1
}}
	(-1)^{e_\alpha(r - \ord_\gamma \alpha)} \\
\intertext{and}
\mterm{\varepsilon\symmram(G/G', \gamma)}
={} & \prodder
	(-1)^{\rk_{\field_{\pm\alpha}} \bG_{\pm\alpha} - 1}
	(-\mf G)^{f_\alpha} \times{} \\
    & \hphantom\prodder\qquad
	\sgn_{\ff_\alpha}(t_\alpha)
	\sgn_{\field_{\pm\alpha}}(\bG_{\pm\alpha}). \\
\intertext{Finally, put}
\mterm{\tilde e(\pi', \gamma)}
={} & \frac
	{\tilde e(G/G', \gamma)}
	{\tilde e(H/H', \gamma)},
\end{align*}}
and similarly for
{\def\amterm#1{%
	\index{notation}{\(#1(\pi', \gamma)\)}%
	\me{#1}%
}
\amterm{\varepsilon\symmunram},
\amterm{\varepsilon\nosymm},
\amterm{\varepsilon\noram},
\amterm{\varepsilon\symmram},
and
\amterm e}.
\end{defn}

The signs occurring in Definition \ref{defn:signs} play an
important r\^ole in the character computations, and their
behaviour is surprisingly subtle (for example,
it necessitates the modification of the original
Yu construction described in
Definition \ref{defn:twisted-cusp-ind} below);
so we take a while to discuss some of their important
properties.

\begin{rem}
The riot of notation in Definition \ref{defn:signs} serves a
purpose; each term is handled differently.
The function
\mapto
	{\varepsilon\noram(\pi')}
	\gamma
	{\varepsilon\noram(\pi', \gamma)}
is a character of \(T\).
We build it into the `twisted' construction
(Definition \ref{defn:twisted-cusp-ind})
of supercuspidals from toral pairs.
The function \(\tilde e(\pi')\)
is the source of the sign changes predicted by Kottwitz
\cite{kottwitz:sign-change};
see Proposition \ref{prop:stable-sign}.
The function \(\varepsilon\symmram(\pi')\)
is a bit more mysterious;
but, by Lemma \ref{lem:stable-sign}, it does not change
under stable conjugacy, and so, for our purposes, this
mystery may be swept under the rug.
\end{rem}

\begin{rem}
If \bT is an elliptic \(G\)-torus and \(\alpha \in \Root\),
then
\(\sum_{\sigma \in \Gamma/\Gamma_\alpha} \alpha\)
is a (rational) character of \(\bT/Z(\bG)\),
hence is trivial;
so
\begin{align*}
\sgn_{\ff_\alpha}(\alpha(\gamma))^{e_\alpha}
& {}= \sgn_\ff\bigl(
	N_{\ff_\alpha/\ff}(\alpha(\gamma))
\bigr)^{e_\alpha} \\
& {}= \sgn_\ff\bigl(
	N_{\field_\alpha/\field}(\alpha(\gamma))
\bigr) \\
& {}= \sgn_\ff\Bigl(
	\bigl(\sum_{\sigma \in \Gamma/\Gamma_\alpha} \alpha\bigr)
	(\gamma)
\Bigr) = 1
\qforall{\(\gamma \in T\).}
\end{align*}
In particular, if \(e_\alpha\) is odd, then
\(\sgn_{\ff_\alpha} \circ \alpha = 1\);
so, if \bT splits over an extension of
odd ramification degree, then the character
\mapto
	{\varepsilon\nosymm_{x, r/2}(G, T)}
	\gamma
	{\varepsilon\nosymm_{x, r/2}(G, T, \gamma)}
of Definition \ref{defn:signs} is trivial.
\end{rem}

\begin{rem}
\label{rem:sign-choices}
The quantities \(\tilde e(G, T, \gamma)\),
\(\varepsilon\symmunram(G, T, \gamma)\),
and \(\varepsilon\symmram(G/G', \gamma)\)
all depend on \(r\), although we suppress it from
the notation for convenience.

By Corollary \ref{cor:opp-ord}, we have that
\(
\tilde e(G, T, \gamma)
= (-1)^
	{\smcard{\Gamma\bslash\Root
		\textsub{\(x, (r - \ord_\gamma)/2\), symm}
	}}\).

The restriction
of a symmetric root \(\alpha\)
to the maximal \field-split subtorus \(\bT\textsub s\) of \bT
is both fixed (because \(\bT\textsub s\) is \field-split)
and negated (because \(\alpha\) is symmetric)
by
some element of \(\Gamma\);
so \(\alpha\) is trivial on \(\bT\textsub s\).
That is, \(\Root(\bG, \bT)\symm = \Root(\bM, \bT)\symm\),
where \(\bM = C_\bG(\bT\textsub s)\).
Since \bT is \field-elliptic in \bM, the image of
\(\BB(\bT, \field)\) in \(\rBB(\bM, \field)\) is a singleton;
so it follows that
\(\tilde e(G, T, \gamma)\),
\(\varepsilon\symmunram(G, T, \gamma)\),
and
\(\varepsilon\symmram(G/G', \gamma)\)
do not depend on the choice of \(x \in \BB(\bT, \field)\).

By Remark \ref{rem:compare-centralizers}, we have that
\(\tilde e(\pi', \gamma)\) depends only on \(\gamma_{< r}\)
(and \(\pi'\), and, possibly, \bT), not on \(\gamma\);
and similarly for \(\varepsilon\symmunram\),
\(\varepsilon\nosymm_{x, r/2}\), \(e\),
\(\varepsilon\symmram\), and \(\varepsilon\noram_{x, r/2}\).

It seems likely that \(\tilde e(G/G', \gamma)\), but \emph{not}
\(\tilde e(G, T, \gamma)\),
is independent of the choice of \tamefield-split maximal
torus \bT in \(\bG'\) containing \(\gamma\),
and similarly for \(\varepsilon\symmunram\),
\(\varepsilon\nosymm_{x, r/2}\),
\(\varepsilon\symmram\),
and \(\varepsilon\noram_{x, r/2}\).
However, we currently do not have a proof of this.
(It is clear that \(e(G/G', \gamma)\) is independent of the
choice of \bT.)
This uncertainty seems to have the potential to affect our
character formul{\ae}
(Theorem \ref{thm:ratl}
and Theorem \ref{thm:almost-stable}).
However, by \xcite{adler-spice:explicit-chars}*
	{Propositions \xref{prop:theta-tilde-phi}
	and \xref{prop:gauss-sum}}
(and \cite{yu:supercuspidal}*{\S11} and
\xcite{adler-spice:explicit-chars}*{Definition 5.2.4}),
at least the product
\(\varepsilon\symmram(\pi', \gamma_{< r})
\varepsilon\noram(\pi', \gamma_{< r})
\tilde e(\pi', \gamma_{< r})\) that occurs in
Theorem \ref{thm:ratl} is independent of \bT;
and, in Theorem \ref{thm:almost-stable}
(and, indeed, throughout \S\ref{sec:stable}),
we are dealing with the case where \(\bG'\) is a torus, so
that there is only one possible choice of \bT.
\end{rem}

All the signs are invariant under rational conjugacy
(Remark \ref{rem:ratl-sign}).
Even better, \(e\) and \(\varepsilon\symmram\) are invariant
under \emph{stable} conjugacy (Lemma \ref{lem:stable-sign}).

\begin{rem}
\label{rem:ratl-sign}
If \(g \in G\), and we put
\begin{itemize}
\item \(x' = g\dota x\),
\item \(\bT' = \Int(g)\bT\),
and
\item \(\gamma' = \Int(g)\gamma\),
\end{itemize}
then it is clear that
\begin{align*}
\tilde e(G, T, \gamma)
	& {}= \tilde e(G, T', \gamma'), \\
\varepsilon\symmunram(G, T, \gamma)
	& {}= \varepsilon\symmunram(G, T', \gamma'), \\
\intertext{and}
\varepsilon\nosymm_{x', r/2}(G, T', \gamma')
	& {}= \varepsilon\nosymm_{x, r/2}(G, T, \gamma).
\end{align*}
\end{rem}

\begin{lem}
\label{lem:stable-sign}
If \(g \in \bG(\unfield)\) is such that
\(g\inv\Frob g \in \bT(\unfield)\), then
\(e(G/G', \gamma) = e(G/G', \Int(g)\gamma)\)
and
\(\varepsilon\symmram(G/G', \gamma)
= \varepsilon\symmram(G/G', \Int(g)\gamma)\).
\end{lem}

\begin{proof}
Put \(\bT' = \Int(g)\bT\) and \(\gamma' = \Int(g)\gamma\).
Note that the group \(\bT'\) and the map \map{\Int(g)}\bT{\bT'}
are defined over \field;
in particular, \(\gamma' \in T'\).
The equality for \(e\) is clear, so we consider only the one
for \(\varepsilon\symmram\).

The element \(g\) induces a natural, \(\Gamma\)-equivariant
bijection \anonmap{\Root(\bG, \bT)}{\Root(\bG, \bT')},
which we denote by \anonmapto\alpha{\alpha'},
such that \(\alpha(\gamma) = \alpha'(\gamma')\)
for all \(\alpha \in \Root(\bG, \bT)\).
Since it is \(\Gamma\)-equivariant, this identification
preserves the notion of symmetry of a root,
and of (un)ramifiedness of a symmetric root.

Choose \(y \in \BB(\bT, \unfield)\)
such that \(g\dota y \in \BB(\bT', \unfield)\)
is \(\Gamma\)-fixed, \ie, lies in \(\BB(\bT', \field)\).
We claim that, for any symmetric and ramified
\(\alpha \in \Root(\bG, \bT)\),
we have \(\ord_x \alpha = \ord_y \alpha\).
Note that \(\alpha\) remains symmetric and ramified
after an \emph{unramified} base change.
Further, by Remark \ref{rem:ord-unram}, the sets
\(\ord_x \alpha\) and \(\ord_y \alpha\) are not affected by
such a change; so we may, and do, assume that
\(x, y \in \BB(\bT, \field) = \AA(\bT_\sepfield, \undefined)^\Gamma\).
Since \(\AA(\bT_\sepfield, \undefined)\) is a torsor under
\(\bX_*(\bT) \otimes_\Z \R\),
where \(\bX_*(\bT)\) is the cocharacter lattice
of \(\bT_\sepfield\),
there is a well defined, \(\Gamma\)-fixed element
\(\lambda \in \bX_*(\bT) \otimes_\Z \R\)
such that \(x + \lambda = y\).
Then
\[
2\pair\alpha\lambda
= \pair{(1 - \sigma_\alpha)\alpha}\lambda
= \pair\alpha{(1 - \sigma_\alpha\inv)\lambda}
= \pair\alpha0 = 0,
\]
as desired.
It is clear that also
\(\ord_y \alpha = \ord_{g\dota y} \alpha'\)
(for \emph{any} root \(\alpha\), whether or not symmetric
and ramified).
Thus, the identification of \(\Root(\bG, \bT)\)
with \(\Root(\bG, \bT')\) establishes a bijection between
\(\Root(\bG, \bT)_
	{\textup{\(x, (r - \ord_\gamma)/2\), symm, ram}}\)
and
\(\Root(\bG, \bT')_
	{\textup{\(g\dota y, (r - \ord_{\gamma'})/2\), symm, ram}}\).
The result follows.
\end{proof}

Proposition \ref{prop:stable-sign} below is the crucial
ingredient in our proof of the stability of certain
character sums (Theorem \ref{thm:stable}).

\begin{prop}
\label{prop:stable-sign}
If \(\bG'\) contains a maximally \unfield-split torus in
\bG, then
\[
\tilde e(\pi', \gamma)
= (-1)^{\rk_\field \bG - \rk_\field \bG'}\dotm
(-1)^{\rk_\field \bH - \rk_\field \bH'}\dotm
e(\pi', \gamma).
\]
\end{prop}

\begin{proof}
We may, and do, assume that \(\gamma = \gamma_{< r}\).

If \(\alpha \not\in \Root' \cup \Root_\bH\),
then we have that
\(r - \ord_\gamma \alpha
= \ord \textup d\alpha^\vee(X^*) - \ord(\alpha(\gamma) - 1)
\in \Z_\alpha\).
By Corollary \ref{cor:ord-symm}, therefore,
\(\frac1 2(r -  \ord_\gamma \alpha) \in \ord_x \alpha\)
if and only if
\(0 \in \ord_x \alpha\)
	and \(r - \ord_\gamma \alpha \in 2\Z_\alpha\),
or
\(0 \not\in \ord_x \alpha\)
	and \(r - \ord_\gamma \alpha \not\in 2\Z_\alpha\).
Thus,
\begin{align*}
\tilde e(\pi', \gamma)
={} & (-1)^
	{\smcard{
		\Gamma\bslash
			(\Root_x \setminus (\Root' \cup \Root_\bH))
	}}
\times{} \\
    & \qquad
	\prod_
	{\alpha \in
		\Gamma\bslash
			(\Root \setminus (\Root' \cup \Root_\bH))
	}
	(-1)^{e_\alpha(r - \ord_\gamma \alpha)},
\end{align*}
where \(\Root_x = \set{\alpha \in \Root}{0 \in \ord_x \alpha}\).
The result now follows from
\xcite{spice:signs-alg}*{Proposition \xref{prop:AS-Levi}}.
\end{proof}

Finally, we relate our roots of unity from
Definition \ref{defn:signs} to those occurring in
\cite{adler-spice:explicit-chars}.
Proposition \ref{prop:root} and its proof may be challenging
to read, because it is comparing two different systems of
notation%
---the \textit{ad hoc} system of
\cite{adler-spice:explicit-chars}, and the
more nearly uniform system of this paper---%
that assign conflicting meanings to some symbols.
We indicate explicitly when we are using the notation
of \cite{adler-spice:explicit-chars}.

\begin{prop}
\label{prop:root}
The quantity
\(\varepsilon(\phi, \gamma)\mf G(\phi, \gamma)\)
(in the notation of
\xcite{adler-spice:explicit-chars}*
	{Propositions
	\xref{prop:theta-tilde-phi}
	and
	\xref{prop:gauss-sum}})
equals
\(\varepsilon\symmram(\pi', \gamma)
\varepsilon\noram(\pi', \gamma)
\tilde e(\pi', \gamma)\)
(in the notation of Definition \ref{defn:signs}).
\end{prop}

\begin{proof}
We may, and do, assume that \(\gamma = \gamma_{< r}\).

Preserve the notation of Definition \ref{defn:signs}.
There are defined in
\xcite{adler-spice:explicit-chars}*
	{Notations \xref{notn:weil} and \xref{notn:gauss}}
disjoint sets
\(\Xi(\phi, \gamma)\) and \(\Upsilon(\phi, \gamma)\)
such that
\(\Xi(\phi, \gamma) \cup \Upsilon(\phi, \gamma)
= \Root(\pi', \gamma)\).
Those roots such that \(\ord_\gamma \alpha = 0\) lie in
\(\Xi(\phi, \gamma)\),
and those such that \(0 < \ord_\gamma \alpha < r\) in
\(\Upsilon(\phi, \gamma)\).
Note in particular that \(\alpha(\gamma) \in 1 + \pp_\alpha\)
whenever \(\alpha \in \Upsilon(\phi, \gamma)\).
We abbreviate \(\Xi(\phi, \gamma)\) to \(\Xi\),
and \(\Upsilon(\phi, \gamma)\) to \(\Upsilon\).

Now note that
\begin{align*}
\varepsilon\symmunram(\pi', \gamma)
& {}= \prod_{\alpha \in \Gamma\bslash
	(\Root\textsub{\(x, r/2\), symm, unram} \setminus \Root')}
	\sgn_{\ff_\alpha^1}(\alpha(\gamma)) \\
& {}= \prod_{\alpha \in \Gamma\bslash\Xi\symmunram}
	\sgn_{\ff_\alpha^1}(\alpha(\gamma))
\dotm\prod_{\substack
	{\alpha \in \Gamma\bslash
		(\Root\textsub{\(x, r/2\), symm, unram} \setminus \Root') \\
	\ord_\gamma \alpha > 0}
}
	\underbrace{\sgn_{\ff_\alpha^1}(\alpha(\gamma))}
		_{(*)} \\
& {}= \prod_{\alpha \in \Gamma\bslash\Xi\symmunram}
	\sgn_{\ff_\alpha^1}(\alpha(\gamma))
\dotm\prod_{\alpha \in \Gamma\bslash\Upsilon\symmunram}
	\underbrace{\sgn_{\ff_\alpha^1}(\alpha(\gamma))}
		_{(*)} \\
& {}= \prod
	_{\alpha \in \Gamma\bslash\Root(\pi', \gamma)\symmunram}
	\sgn_{\ff_\alpha^1}(\alpha(\gamma)),
\end{align*}
since each multiplicand (\(*\)) equals \(1\).
Similarly,
\[
\varepsilon\nosymm(\pi', \gamma)
= \prod_{\alpha \in \Gamma\bslash\Root(\pi', \gamma)\nosymm}
	\sgn_{\ff_\alpha}(\alpha(\gamma)).
\]
We actually work in the proof with these products,
rather than the ones appearing in Definition \ref{defn:signs}.

The proof of \xcite{adler-spice:explicit-chars}*
	{Proposition \xref{prop:theta-tilde-phi}}
uses \cite{gerardin:weil}*{Theorem 4.9.1}.
As mentioned there, one could instead use
\cite{gerardin:weil}*{Corollary 4.8.1}; we do so here.
Recall the decomposition
\[
\mc V
= \mc V^{({0+})} \oplus
\bigoplus_{\alpha \in \Gamma\bslash\Xi\symm}
	V_\alpha \oplus
\bigoplus_{\alpha \in {\pm\Gamma}\bslash\Xi\nosymm}
	V_{\pm\alpha}
\]
\xcite{adler-spice:explicit-chars}*{\xeqref{eq:weil-V-decomp}}.
Then, in the notation of
\xcite{adler-spice:explicit-chars}*
	{Proposition \xref{prop:theta-tilde-phi}}
(and its proof),
the quantity denoted there by
\(\varepsilon(\phi, \gamma)\) is the angular component of
the complex number
\begin{equation}
\tag{$\dag$}
\label{eq:weil-factor}
\theta_{W^{\mc V^{({0+})}}_\zeta}(\gamma)\dotm
\prod_{\alpha \in \Gamma\bslash\Xi\symm}
	\theta_{W^{V_\alpha}_\zeta}(\gamma)\dotm
\prod_{\alpha \in {\pm\Gamma}\bslash\Xi\nosymm}
	\theta_{W^{V_{\pm\alpha}}_\zeta}(\gamma)
\end{equation}
\xcite{adler-spice:explicit-chars}*
	{proof of Proposition \xref{prop:theta-tilde-phi},
	p.~1159, (\dag)}.
By \xcite{adler-spice:explicit-chars}*
	{\xeqref{eq:weil-trivial} and \xeqref{eq:weil-non-symm}},
\(\theta_{W^{\mc V^{({0+})}}_\zeta}(\gamma)\) is a positive
number,
and
\begin{equation}
\tag{$\dag\dag$}
\label{eq:nosymm-sign}
\prod_{\alpha \in {\pm\Gamma}\bslash\Xi\nosymm}
	\theta_{W^{V_{\pm\alpha}}_\zeta}(\gamma)
= \prod_{\alpha \in {\pm\Gamma}\bslash\Xi\nosymm}
	\sgn_{\ff_\alpha}(\alpha(\gamma))
= \varepsilon\nosymm(\pi', \gamma),
\end{equation}
where the notation on the right is that of
Definition \ref{defn:signs}
(and we have used that
\(\alpha(\gamma)\) is in \(1 + \pp_\alpha\)
for \(\alpha \in \Upsilon\)).
Thus, to understand the quantity
\(\varepsilon(\phi, \gamma)\) of 
\cite{adler-spice:explicit-chars},
it remains only to compute the angular component of
\(\theta_{W^{V_\alpha}_\zeta}(\gamma)\)
for \(\alpha \in \Xi\symm\).

If \(\alpha \in \Xi\symmram\),
then, as in \xcite{adler-spice:explicit-chars}*
	{proof of Proposition \xref{prop:theta-tilde-phi},
	p.~1159},
we obtain a non-degenerate pairing
\anonmapto
	{t_1 \otimes t_2}
	{e_\alpha\tr_{\ff_\alpha/\mathbb F_p}
		c_\alpha\dotm t_1\sigma_\alpha(t_2)}
on \(\ff_\alpha\), where
\(c_\alpha \in \ff_\alpha\)
satisfies \(\sigma_\alpha c_\alpha = -c_\alpha\).
Since \(\sigma_\alpha \in \Gamma\unfix\) acts trivially on
\(\ff_\alpha\), and \(p \ne 2\),
we have that \(c_\alpha = 0\).
Therefore, the non-degenerate pairing above is identically \(0\),
so its domain \(V_\alpha\) is \(0\).
By definition
\xcite{adler-spice:explicit-chars}*
	{Notation \xref{notn:weil}},
we have that
\(V_\alpha = \pmb\fg_\alpha(\field_\alpha)_{x, r/2:{(r/2)+}}\);
so, by Definition \ref{defn:ord-x}, this means that
\(r/2 = \tfrac1 2(r - \ord_\gamma \alpha)\)
does \emph{not} lie in \(\ord_x \alpha\),
which is a contradiction.

Thus \(\Xi\symm = \Xi\symmunram\).
Choose \(\alpha \in \Xi\symmunram\),
so that \(f_\alpha = 2f_{\pm\alpha}\).
There is an \ff-linear isomorphism
\begin{equation}
\tag{$\dag\dag\dag$}
\label{eq:weil-factor-iso}
V_\alpha \cong \ff_\alpha
\end{equation}
that identifies the symplectomorphism
\(\gamma_\alpha\) of \(V_\alpha\) induced by \(\gamma\)
with multiplication by \(\alpha(\gamma)\);
and
there is an \(\ff_\alpha\)-split \ff-torus \(\ms S_\alpha\)
in \(\Sp_{V_\alpha}\) such that an element
\(\ol s \in \Sp_{V_\alpha}(\ff)\) lies in
\(\ms S_\alpha(\ff)\) if and only if it is identified
\via \eqref{eq:weil-factor-iso} with
multiplication by an element of
\(\ff_\alpha^1\).
In particular, \(\gamma_\alpha \in \ms S_\alpha(\ff)\).

In order to apply \cite{gerardin:weil}*{Corollary 4.8.1},
we need to understand the \(\pm\Gamma\)-orbits of weights of
the action of \(\ms S_\alpha\) on \(V_\alpha\).
Since \(\sigma_\alpha\) acts on
the character lattice \(\bX^*(\ms S_\alpha)\)
of \((\ms S_\alpha)_{\ff\sep}\)
by negation, every such weight is symmetric.
Now, if \(\Omega\) is any orbit of weights,
then, in the notation of
\cite{gerardin:weil}*{Definition 4.6},
\(i(\Omega) = \tfrac1 2\card\Omega\) is the least
positive integer \(i\)
such that \(\Frob^i\) acts by negation on \(\Omega\).
Now
\(\Frob^{f_{\pm\alpha}} \in \sigma_\alpha\Gamma\unfix\),
so that \(\Frob^{f_{\pm\alpha}}\) and \(\sigma_\alpha\)
act the same on \(\bX^*(\ms S_\alpha)\).
It follows that \(f_{\pm\alpha}\) is an odd multiple of
\(i(\Omega)\),
so that \(f_\alpha = 2f_{\pm\alpha}\)
is an odd multiple of \(\card\Omega = 2i(\Omega)\)
(in particular, \(f_\alpha\) and \(\card\Omega\) have the
same parity)
and
\(q_{\pm\alpha} = q^{f_{\pm\alpha}}\)
is an odd power of
\(q_\Omega \ldef q^{i(\Omega)}\).
This means that the character \(\chi_\Omega^{\ms S_\alpha}\)
of \cite{gerardin:weil}*{Lemma 4.6(2)(e)} is
\[
\ol s \mapsto \epsilon(\ol s)^{(1 + q_\Omega)/2}
= \epsilon(\ol s)^{(1 + q_{\pm\alpha})/2}
= \sgn_{\ff_\alpha^1}(\epsilon(\ol s)),
\]
where \(\epsilon\) is any element of \(\Omega\).
We have used that
\(\epsilon(\ol s)^{(1 + q_\Omega)/2}\)
	is in \(\sset{\pm1}\)
for all \(\ol s \in \ms S_\alpha(\ff)\),
and that
\(q_{\pm\alpha} + 1
= \smcard{\ff_\alpha^1}\).
Finally, since there are
\(\dim_\ff V_\alpha = f_\alpha\)
weights of \(\ms S_\alpha\) on \(V_\alpha\),
and the size of each orbit has the same parity as
\(f_\alpha\),
the number \(l(V_\alpha, \ms S_\alpha; \gamma_\alpha)\)
of orbits of weights is odd.
Since \eqref{eq:weil-factor-iso} identifies
\(\gamma_\alpha\) with multiplication by the
projection to \(\ff_\alpha\) of
\(\alpha(\gamma) \in \pint_\alpha\),
we have that
\(N(V_\alpha; \gamma_\alpha)
\ldef \frac1 2\dim \ker \res(\gamma_\alpha - 1)to{V_\alpha}
= 0\)
and, for any weight \(\epsilon\), that
\(\gamma_\alpha^\epsilon\) is a Galois conjugate of
\(\alpha(\gamma)\).
Thus, \cite{gerardin:weil}*{Corollary 4.8.1} gives
\begin{align*}
\theta_{W^{V_\alpha}_\zeta}(\gamma)
& = (-1)^{l(V_\alpha, \ms S_\alpha; \gamma_\alpha)}
q^{N(V_\alpha; \gamma_\alpha)}
\prod_\Omega \chi^{\ms S_\alpha}_\Omega(\gamma_\alpha) \\
& = \bigl(
	-\sgn_{\ff_\alpha^1}(\alpha(\gamma))
\bigr)^{l(V_\alpha, \ms S_\alpha; \gamma_\alpha)} \\
& = -\sgn_{\ff_\alpha^1}(\alpha(\gamma)).
\end{align*}
That is, by \eqref{eq:weil-factor} and
\eqref{eq:nosymm-sign},
\begin{equation}
\tag{$\ddag_0$}
\label{eq:shallow-sign}
\begin{aligned}
\varepsilon(\phi, \gamma)
& {}= (-1)^{\card{\Gamma\bslash\Xi\symm}}\dotm
\varepsilon\nosymm(\pi', \gamma)\dotm
\prod_{\alpha \in \Gamma\bslash\Xi\symmunram}
	\sgn_{\ff_\alpha^1}(\alpha(\gamma)) \\
& {}= (-1)^{\card{\Gamma\bslash\Xi\symm}}\dotm
\varepsilon\nosymm(\pi', \gamma)\dotm
\varepsilon\symmunram(\pi', \gamma) \\
& {}= (-1)^{\card{\Gamma\bslash\Xi\symm}}\dotm
\varepsilon\noram(\pi', \gamma),
\end{aligned}
\end{equation}
where the notation on the left is that of
\xcite{adler-spice:explicit-chars}*
	{Proposition \xref{prop:theta-tilde-phi}},
and that on the right is that of
Definition \ref{defn:signs}
(and we have used again that
\(\alpha(\gamma) \in 1 + \pp_\alpha\)
for \(\alpha \in \Upsilon\)).

Now, since
\(\Root(\pi', \gamma)\symmram = \Upsilon\symmram\),
\xcite{adler-spice:explicit-chars}*
	{Proposition \xref{prop:gauss-sum}}
says precisely that
\begin{equation}
\tag{$\ddag_{> 0}$}
\label{eq:deep-sign}
\mf G(\phi, \gamma)
= (-1)^{\card{\Gamma\bslash\Upsilon\symm}}
\varepsilon\symmram(\pi', \gamma),
\end{equation}
where the notation on the left is that of
\xcite{adler-spice:explicit-chars}*
	{Proposition \xref{prop:gauss-sum}},
and that on the right is that of
Definition \ref{defn:signs}.

The desired equality now follows from
\eqref{eq:shallow-sign} and \eqref{eq:deep-sign}.
\end{proof}

\subsection{An indexing set}
\label{sec:index}

In this section and \S\ref{sec:orbital} below,
we change notation, putting
\(\bH = C_\bG(\gamma_{< r})\)
	(not \(\bH = \CC\bG r(\gamma) = C_\bG(\gamma_{< r})\conn\))
and
\(\bH' = C_{\bG'}(\gamma_{< r})\).
The reasoning of \xcite{adler-spice:good-expansions}*
	{Proposition \xref{prop:unique-approx}}
guarantees only that
\(\bH\conn\), not \bH itself, is independent of the choice
of approximation \(\gamma = \gamma_{< r}\gamma_{\ge r}\);
but we soon reduce all relevant statements about \bH
to statements about \(\bH\conn\),
so it doesn't matter.

The following set occurs implicitly in
\xcite{adler-spice:explicit-chars}*
	{Theorem \xref{thm:char-tau|pi-1}},
where it is the indexing set for a sum used in the
computation of the character of \(\YuUp_{G'}^G \pi'\).

\begin{defn}
\label{defn:index}
Put
\(
\mterm{\mc C(\pi', \gamma)}
\ldef \sett
	{g \in G}
	{\(\Int(g)\gamma_{< r} \in G'\)
		and \(\Int(g)\gamma_{\ge r} \in G_{x, r}\)}
\).
\end{defn}

Write \mc C for \(\mc C(\pi', \gamma)\).
Lemma \ref{lem:rigidity} below is a technical result on the
indexing set, proven in \cite{adler-spice:good-expansions};
we are mostly interested in its consequences,
Corollaries \ref{cor:rigid-sum}
and \ref{cor:rigid-intersect}.

\begin{lem}[\xcite{adler-spice:good-expansions}*
	{Lemma \xref{lem:rigidity}}]
\label{lem:rigidity}
\(\mc C \cap G_{x, {0+}}
= G'_{x, {0+}}H_{x, {0+}}\).
\end{lem}

\begin{cor}
\label{cor:rigid-sum}
The natural projection
\[
\anonmap
	{\stab_{G'}(\ox)\bslash\mc C/H}
	{\stab_{G'}(\ox)G_{x, {0+}}\bslash
		G_{x, {0+}}\mc C/H}
\]
is a bijection.
\end{cor}

\begin{proof}
Let us call the map \(i_\gamma\).
It is clear that it is a surjection.

Fix \(g \in \mc C\).
Replacing \(\gamma\) by \(\Int(g)\gamma\)
(hence \(\gamma_{< r}\) by \(\Int(g)\gamma_{< r}\)
and \(\gamma_{\ge r}\) by \(\Int(g)\gamma_{\ge r}\))
replaces
	\bH by \(\Int(g)\bH\),
	\(\mc C\) by \(\mc C g\inv\),
and	\(i_\gamma\) by \(c g\inv \mapsto i_\gamma(c)g\inv\)%
;
so it suffices to compute the fibre over \(g = 1\).
If \(c \in \mc C\) is such that
\(\stab_{G'}(\ox)G_{x, {0+}}c H
= i_\gamma(c) = i_\gamma(1)
= \stab_{G'}(\ox)G_{x, {0+}}H\), then,
since \(G_{x, {0+}}\) is normalised by
\(\stab_{G'}(\ox)\), there exists
\(k \in G_{x, {0+}} \in \stab_{G'}(\ox)c H\).
Such an element \(k\) lies in
\(\mc C(\pi', \gamma_{< r})\)
(not necessarily \(\mc C(\pi', \gamma)\)).
By Lemma \ref{lem:rigidity}, it follows that
\(k \in G'_{x, {0+}}H_{x, {0+}} \subseteq \stab_{G'}(\ox)H\),
so that
\(\stab_{G'}(\ox)k H = \stab_{G'}(\ox)c H\)
is the trivial double coset.
\end{proof}

\begin{cor}
\label{cor:rigid-intersect}
If \(\gamma_{< r} \in G'\)
,
then
\begin{multline*}
\stab_{G'}(\ox)G_{x, {0+}} \cap H
= \stab_{H'}(\ox)H_{x, {0+}} \\
\andq
\stab_{G'}(\ox)G_{x, {0+}} \cap H\conn
= \stab_{H^{\prime\,\circ}}(\ox)H_{x, {0+}}.
\end{multline*}
\end{cor}

\begin{proof}
%
Suppose that \(h \in \stab_{G'}(\ox)G_{x, {0+}} \cap H\),
and write \(h = g'k\),
where \(g' \in \stab_{G'}(\ox)\)
and \(k \in G_{x, {0+}}\).
As in the proof of Corollary \ref{cor:rigid-sum},
we have that
\(k \in \stab_{G'}(\ox)H \subseteq
	\mc C(\pi', \gamma_{< r}) \cap \),
hence by Lemma \ref{lem:rigidity} that
\(k \in G'_{x, {0+}}H_{x, {0+}}\).
It follows that
\[
h = g'k
	\in H \cap \stab_{G'}(\ox)H_{x, {0+}}
	= \bigl(H \cap \stab_{G'}(\ox)\bigr)H_{x, {0+}}
	= \stab_{H'}(\ox)H_{x, {0+}},
\]
whence the first containment.

For the second containment, observe that,
since \(H_{x, {0+}} \subseteq H\conn\)
and \(H' \cap H\conn = G' \cap H\conn = H^{\prime\,\circ}\)
(see Remark \ref{rem:funny-centralizer-descends}),
we have
\begin{multline*}
H\conn \cap \stab_{G'}(\ox)G_{x, {0+}}
= H\conn \cap \stab_{H'}(\ox)H_{x, {0+}} \\
= \bigl(H\conn \cap \stab_{H'}(\ox)\bigr)H_{x, {0+}}
= \stab_{H^{\prime\,\circ}}(\ox)H_{x, {0+}}.\qedhere
\end{multline*}
\end{proof}

\subsection{Orbital integrals}
\label{sec:orbital}

Preserve the notation of \S\ref{sec:index}.
In particular, \bH and \(\bH'\) are the possibly
\emph{disconnected} groups
\(C_\bG(\gamma_{< r})\) and \(C_{\bG'}(\gamma_{< r})\),
respectively.
Proposition \ref{prop:orbit} relates orbital
integrals on \(H\) and \(H\conn\); but note that, as in
Proposition \ref{prop:root}, in addition to changing the
domain of integration, we are also changing the
normalisation of the measure with respect to which it is
computed.

\begin{prop}
\label{prop:orbit}
We have that
\begin{align*}
\hat\mu^{\stab_H(\ox)}_{X^*}
={}	&
\indx{\stab_{H'}(\ox)}{\stab_{H^{\prime\,\circ}}(\ox)}\inv
\card{(\fh', \fh)_{x, (0, 0):(0, {0+})}}^{1/2}\times{} \\
& \qquad\sum_{h \in \stab_H(\ox)/\stab_{H\conn}(\ox)}
	\hat\mu^{\stab_{H\conn}(\ox)}_{\Ad^*(h)X^*}
\end{align*}
and
\begin{align*}
\hat\mu^H_{X^*}
={}	&
\indx{H'}{H^{\prime\,\circ}}\inv
\card{(\fh', \fh)_{x, (0, 0):(0, {0+})}}^{1/2}\times{} \\
& \qquad\sum_{h \in H/H\conn}
	\hat\mu^{H\conn}_{\Ad^*(h)X^*},
\end{align*}
as functions on \(\fh\rss\),
where the orbital integrals on the left are normalised
as in \xcite{adler-spice:explicit-chars}*
	{Corollary \xref{cor:char-tau|pi-1}},
and those on the right as in
Definition \ref{defn:normal-harm}.
\end{prop}

See \xcite{adler-spice:explicit-chars}*
	{\xeqref{eq:mu-stab-ox}}
for the definition of
\(\hat\mu^{\stab_H(\ox)}_{X^*}\).
We do not reproduce it here, since we are really interested
only in \(\hat\mu^H_{X^*}\).

\begin{proof}
Since the arguments are essentially identical, we discuss
only the integral over \(H\).

As observed after the proof of
\xcite{adler-spice:explicit-chars}*
	{Theorem \xref{thm:char-tau|pi-1}},
we may use \xcite{adler-spice:explicit-chars}*
	{Lemma \xref{lem:disc-orbital-int}}
to re-write the Fourier transform of an orbital integral
with respect to \(H\),
which is the group of \field-rational points of a possibly
disconnected group,
as a sum of Fourier transforms of orbital integrals with
respect to \(H\conn\),
which is the group of \field-rational points of a connected
group.

Since \(C_H(X^*) = H'\), we have by
\xcite{adler-spice:explicit-chars}*
	{Lemma \xref{lem:disc-orbital-int}}
that
\[
\hat\mu^H_{X^*}
= \indx{H'}{H^{\prime\,\circ}}\inv
\sum_{h \in H/H\conn}
	\hat\mu^{H\conn}_{\Ad^*(h)\inv X^*}.
\]

Note, however, that the right-hand side is still normalised
with respect to the measure on
\(H\conn/H^{\prime\,\circ}\)
inherited from its embedding in
\(H/H'\).
By Corollary \ref{cor:rigid-intersect}
(and \xcite{adler-spice:explicit-chars}*
	{Corollary \xref{cor:char-tau|pi-1}
	and \S\xref{sec:JK}}),
this measure assigns mass \(1\) to
\[
\bigl(H \cap \stab_{G'}(\ox)G_{x, {0+}}\bigr)H'/H'
= H_{x, {0+}}H'/H'
\]
---\ie, to the image in \(H/H'\) of
\(H_{x, {0+}} \subseteq H\conn\),
which lies in (the image in \(H/H'\) of)
\(H\conn/H^{\prime\,\circ}\).

On the other hand, since
\(H^{\prime\,\circ} \cap H_{x, {0+}} = H'_{x, {0+}}\)
by \xcite{adler-spice:good-expansions}*
	{Lemma \xref{lem:funny-centralizer-descends}
	and Corollary \xref{cor:compatibly-filtered-tame-rank}},
we have that the quotient of the canonical measures
on \(H\conn\) and \(H^{\prime\,\circ}\)
(see Definition \ref{defn:normal-harm})
assigns to the image in \(H\conn/H^{\prime\,\circ}\)
of \(H_{x, {0+}}\)
the quotient of the canonical measure
\[
\indx{H_{x, 0}}{H_{x, {0+}}}\inv\dotm
	\bigl(
		\card{H_{x, 0:{0+}}}\dotm\card{\fh_{x, 0:{0+}}}^{-1/2}
	\bigr)
= \card{\fh_{x, 0:{0+}}}^{-1/2}
\]
of \(H_{x, {0+}}\)
by the canonical measure \(\card{\fh'_{x, 0:{0+}}}^{-1/2}\)
of \(H'_{x, {0+}}\).
The result follows.
\end{proof}

\subsection{An inductive formula for normalised characters}
\label{sec:char-char}

As preparation for our character computation
Theorem \ref{thm:ratl}, which is the main result of this
section, we present the complicated-looking
Lemma \ref{lem:char-summand},
which just assembles all of our previous results on the
comparison between the notations
of \cite{adler-spice:explicit-chars}
and of this paper.
As with Proposition \ref{prop:root}, it is important to keep
in mind that there are two different systems of notation in
play.

Recall that we have required that our expanded cuspidal
quintuple \((G', \pi', \phi, x, X^*)\) is compact
(\ie, that the group \(\bG'/Z(\bG)\) is \field-anisotropic),
so that the character \(\Theta_{\pi'}\),
hence also the function \(\Phi_{\pi'}\) of
Definition \ref{defn:normal-harm},
is defined on all of \(G'\).

\begin{lem}
\label{lem:char-summand}
If \(\gamma_{< r} \in G'\)
and \(Y \in \Lie(\CC G r(\gamma))\),
then we have that
\begin{multline*}
\underbrace{\varepsilon(\phi, \gamma)
\mf G(\phi, \gamma)}
	\textsub{(I)}\times{} \\
\bigindx
	{\odc{\gamma; x, r}}
	{\odc{\gamma; x, r}_{G'}G_{x, r/2}}^{1/2}
\bigindx
	{\odc{\gamma_{< r}; x, {r+}}}
	{\odc{\gamma_{< r}; x, {r+}}_{G'}G_{x, {(r/2)+}}}^{1/2}\times{} \\
\underbrace{\Theta_{\pi'}(\gamma_{< r})}
	\textsub{(II)}
\underbrace{\hat\mu^{\CC G r(\gamma)}_{X^*}(Y)}
	\textsub{(III)}
\end{multline*}
(in the notation of
\xcite{adler-spice:explicit-chars}*
	{Propositions \xref{prop:theta-tilde-phi}
and \xref{prop:gauss-sum}}
and \xcite{adler-spice:good-expansions}*
	{Definition \xref{defn:bracket}},
and with the normalisation of measure of
\xcite{adler-spice:explicit-chars}*
	{Corollary \xref{cor:char-tau|pi-1}})
equals
\begin{multline*}
\abs{\redD_G(\gamma)}^{-1/2}
\smabs{\redD_{\CC G r(\gamma)}(Y)}^{-1/2}
\times{}\\
\underbrace{\indx{C_{G'}(\gamma_{< r})}{\CC{G'} r(\gamma)}\inv}
	\textsub{(\(\textup{III}'_1\))}
\underbrace{\varepsilon\symmram(\pi', \gamma)
	\varepsilon\noram(\pi', \gamma)}
	\textsub{(\(\textup I'_1\))}
\underbrace{\Phi_{\pi'}(\gamma)}
	\textsub{(\(\textup{II}'\))}
\times{} \\
\underbrace{\tilde e(\pi', \gamma)}
	\textsub{(\(\textup I'_2\))}
\underbrace{\sum_{h \in C_G(\gamma_{< r})/\CC G r(\gamma)}
	\hat O^{\CC G r(\gamma)}_{\Ad^*(h)\inv X^*}(Y)}
	\textsub{(\(\textup{III}'_2\))}
\end{multline*}
(in the notation of
Definitions
	\ref{defn:disc},
	\ref{defn:normal-harm},
	and
	\ref{defn:signs}).
\end{lem}

\begin{proof}
As in \S\S\ref{sec:index}--\ref{sec:orbital},
put \(\bH = C_\bG(\gamma_{< r})\)
and \(\bH' = C_{\bG'}(\gamma_{< r})\),
so that \(\CC\bG r(\gamma) = \bH\conn\)
and \(\CC{\bG'}r(\gamma) = \bH^{\prime\,\circ}\)
\xcite{adler-spice:good-expansions}*
	{Corollary \xref{cor:compare-centralizers}}.
(As in \S\ref{sec:char-notn},
we have abbreviated \(\CC{\bG'}r(\gamma_{< r})\)
to \(\CC{\bG'}r(\gamma)\).)

By Proposition \ref{prop:root},
\((\textup I) = (\textup I'_1)\dotm(\textup I'_2)\).
By Definition \ref{defn:normal-harm},
\((\textup{II})
= \abs{\redD_{G'}(\gamma_{< r})}^{-1/2}(\textup{II}')\).
By Proposition \ref{prop:orbit},
\[
(\textup{III})
= \card{(\fh', \fh)_{x, (0, 0):(0, {0+})}}^{1/2}
\smabs{\redD_H(X^*)}^{-1/2}
\smabs{\redD_H(Y)}^{-1/2}
(\textup{III}'_1)\dotm(\textup{III}'_2).
\]
The result then follows from Corollary \ref{cor:const}.
\end{proof}

We are now ready to re-write the character formul{\ae} of
\cite{adler-spice:explicit-chars} in what seems to us to be
nearly the optimal form.
With the exception of Hypothesis \ref{hyp:stronger-mock-exp},
which has not been needed so far,
the hypotheses of Theorem \ref{thm:ratl} are the same ones
that have been in force throughout the section;
we simply re-capitulate them here for convenient reference.
Remember that we require \bG to satisfy
\xcite{adler-spice:good-expansions}*
	{Hypotheses \xref{hyp:conn-cent}
and \xref{hyp:good-weight-lattice}}.

\begin{thm}
\label{thm:ratl}
Suppose that the group \bG satisfies
Hypothesis \ref{hyp:stronger-mock-exp}
(in addition to
\xcite{adler-spice:good-expansions}*
	{Hypotheses \xref{hyp:conn-cent}
and \xref{hyp:good-weight-lattice}}).
Let \((G', \pi', \phi, x, X^*)\) be a depth-\(r\), compact,
central, expanded cuspidal quintuple in \(G\).
Suppose that \(\gamma = \gamma_{< r}\gamma_{\ge r}\)
is regular, semisimple, and \(r\)-approximable,
and put \(Y_{\ge r} = \mexp_{\bT, x}\inv \gamma_{\ge r}\) for
some \tamefield-split maximal \(G\)-torus \bT such that
\(\gamma_{< r} \in T\) and \(x \in \BB(\bT, \field)\).
Then
{\def\summer{%
\sum_{\substack{
	(L, \chi) \in \CC G r(\gamma)\bslash G\dota(G', \pi') \\
	\gamma_{< r} \in L
}}%
}
\begin{align*}
\Phi_{\YuUp_{G'}^G \pi'}(\gamma)
={} & \summer
	\varepsilon\symmram(\chi, \gamma_{< r})\dotm
   	\varepsilon\noram(\chi, \gamma_{< r})
	\Phi_\chi(\gamma_{< r})\times{} \\
    & \hphantom\summer\qquad
	\tilde e(\chi, \gamma_{< r})
	\hat O^{\CC G r(\gamma)}_\chi(Y_{\ge r}).
\end{align*}}
\end{thm}

\begin{proof}
As in the proof of Lemma \ref{lem:char-summand},
put \(\bH = C_\bG(\gamma_{< r})\).  Recall that
the connected group \(\CC\bG r(\gamma)\) is the identity
component of \bH
\xcite{adler-spice:good-expansions}*
	{Corollary \xref{cor:compare-centralizers}}.

For every \(g \in G\), we have that
\(g\inv\dota(G', \pi', \phi, x, X^*)\) is a central,
expanded cuspidal \(G\)-quintuple.
Thus, if \(g_1, g_2 \in G\) satisfy
\(g_1\inv\dota(G', \pi') = g_2\inv\dota(G', \pi')\),
then, by Remark \ref{rem:uniq-expand}, we have that
\(\Ad^*(g_1)\inv X^* \equiv \Ad^*(g_2)\inv X^*
	\pmod{\fg_{x, {(-r)+}}}\).
By \cite{yu:supercuspidal}*{Lemma 8.3},
this implies that \(g_1 \in G'g_2\).
Thus, the map
\anonmap{G'\bslash G}{G\dota(G', \pi')}
given by \anonmapto{G' g}{g\inv\dota(G', \pi')}
is a bijection.
Since this map is equivariant for the natural (right)
\(G\)-action on both spaces, it induces a bijection
\anonmap{G'\bslash G/H\conn}{H\conn\bslash G\dota(G', \pi')}.

By \xcite{adler-spice:explicit-chars}*
	{Theorem \xref{thm:char-tau|pi-1}
	and Proposition \xref{prop:induction1}}
and Corollary \ref{cor:rigid-sum},
using the notation of Definition \ref{defn:index},
we have that
\(\Theta_{\YuUp_{G'}^G \pi'}(\gamma)\) is the sum over
\(g \in G'\bslash\mc C(\pi', \gamma_{< r})/H\)
of the quantity denoted in
\cite{adler-spice:explicit-chars} by
\begin{multline*}
\varepsilon(\phi, \delta_{< r})
\mf G(\phi, \delta_{< r})\times{} \\
\bigindx
	{\odc{\delta_{< r}; x, r}}
	{\odc{\delta_{< r}; x, r}_{G'}G_{x, r/2}}^{1/2}
\bigindx
	{\odc{\delta_{< r}; x, {r+}}}
	{\odc{\delta_{< r}; x, {r+}}_{G'}G_{x, {(r/2)+}}}^{1/2}\times{} \\
\Theta_{\pi'}(\delta_{< r})
\hat\mu^{C_G(\delta_{< r})}_{X^*}(
	\mexp_{\bT, x}\inv \delta_{\ge r}
),
\end{multline*}
where \(\delta = \Int(g)\gamma\)
(so that \(\delta_{< r} = \Int(g)\gamma_{< r}\)
and \(\delta_{\ge r} = \Int(g)\gamma_{\ge r}\)).
(Our \(\pi'\) corresponds to \(\pi' \otimes \phi\),
in the notation of \cite{adler-spice:explicit-chars}.)
Note that \(\delta_{< r} = \Int(g)\gamma_{< r}\)
and \(\delta_{\ge r} = \Int(g)\gamma_{\ge r}\).
By Lemmata \ref{lem:char-summand} and \ref{lem:part-disc},
Remark \ref{rem:ratl-sign},
and Hypothesis \ref{hyp:stronger-mock-exp},
we may re-write the summand as
\begin{multline}
\tag{$*$}
\label{thm:ratl:eq:summand}
\smabs{\redD_G(\gamma)}^{-1/2}\dotm
\varepsilon\symmram(\chi, \gamma_{< r})\dotm
\varepsilon\noram(\chi, \gamma_{< r})
	\Phi_\chi(\gamma_{< r})\times{} \\
\indx{L \cap H}{(L \cap H)\conn}\inv
\sum_{h \in H/H\conn}
	\tilde e(\chi, \gamma_{< r})
	\hat O^{H\conn}_{\Ad^*(h)\inv X^{\prime\,*}}(Y_{\ge r}),
\end{multline}
where
\((L, \chi, X^{\prime\,*}) = g\inv\dota(G', \pi', X^*)\).

Now write
\(\wtilde{\mc C}
= \set{g \in G}{\Int(g)\gamma_{< r} \in G'}\).
For any \(g \in \wtilde{\mc C}\), we have that
\(g\inv\dota(G', \pi', \phi, x, X^*)\) is a central,
expanded cuspidal \(G\)-quintuple, hence,
with the notation of Definition \ref{defn:O-pi}, that
\begin{equation}
\tag{$**$}
\label{thm:ratl:eq:O-chi}
\hat O^{H\conn}_{\Ad^*(g)\inv X^*}
= \hat O^{H\conn}_{\pi' \circ \Int(g)}.
\end{equation}
Further, the expression \eqref{thm:ratl:eq:summand} is
defined, and, by \xcite{adler-spice:explicit-chars}*
	{Lemma \xref{lem:orbital-cancel}
	and Hypothesis \xref{hyp:strong-mock-exp-equi}},
equals \(0\) if \(g \not\in \mc C\);
so \(\Theta_\pi(\gamma)\) is also the sum of
\eqref{thm:ratl:eq:summand} over \(\wtilde{\mc C}\).

Now an elementary counting argument, together with
Remark \ref{rem:funny-centralizer-descends},
shows that
\[
\sum_{g \in G'\bslash G/H}
	\indx
		{\Int(g)\inv G' \cap H}
		{(\Int(g)\inv G' \cap H)\conn}\inv
\sum_{h \in H/H\conn} F(g h)
= \sum_{g \in G'\bslash G/H\conn} F(g)
\]
for any \map F{G'\bslash G/H\conn}\C; so combining
\eqref{thm:ratl:eq:summand} and \eqref{thm:ratl:eq:O-chi},
and using the bijection
\anonmap{G'\bslash G/H\conn}{H\conn\bslash G\dota(G', \pi')},
gives the desired result.
\end{proof}


%
%

\numberwithin{thm}{section}

\section{Stable character sums}
\label{sec:stable}

Throughout this section, fix a depth-\(r\), toral,
cuspidal pair \((T, \phi)\).
This pair plays the same r\^ole as the pair
\((G', \pi')\) did in \S\ref{sec:ratl}; but
note that the requirements that the pair
be compact (Definition \ref{defn:cusp-pair})
and central (Definition \ref{defn:cent-cusp-pair})
are now redundant.
We will soon require
(in Theorems \ref{thm:almost-stable} and \ref{thm:stable})
that \bT is \unfield-split, but we do not need to do so yet.

Also fix, until Theorem \ref{thm:stable},
a regular, semisimple, \(r\)-approximable element
\(\gamma = \gamma_{< r}\gamma_{\ge r}\),
and put \(\bH = \CC\bG r(\gamma)\)
(\ie, not the possibly disconnected group
\(C_\bG(\gamma_{< r})\) of \S\ref{sec:index},
but its identity component
\xcite{adler-spice:good-expansions}*
	{Corollary \xref{cor:compare-centralizers}},
as in \S\ref{sec:char-notn}).

We are almost in a position to show that certain
character sums are stable (see Theorem \ref{thm:stable}),
but we need two technical preliminaries.
First, we show in Lemma \ref{lem:SO} that the definition of
\(\smash{\widehat{SO}}^H_{X^*}\) as a sum
(Definition \ref{defn:normal-harm})
may be translated into an analogous definition of
\(\smash{\widehat{SO}}^H_\phi\)
(Definition \ref{defn:O-pi}).
Next, we show in Lemma \ref{lem:stab-approx} that normal
approximations behave well under stable conjugacy.

\begin{lem}
\label{lem:SO}
Write \(\bH\dota(T, \phi)\) for the set of
stable \bH-conjugates of \((T, \phi)\)
\cite{debacker-reeder:depth-zero-sc}*{\S9.4, p.~850}.
Then
\[
\smash{\widehat{SO}}^H_\phi
= \sum_{(S, \theta) \in H\bslash\bH\dota(T, \phi)}
	\hat O^H_\theta
\]
(as functions on \(\fh\rss \cap \fh_r\)).
\end{lem}

\begin{proof}
Recall that \(X^*\) is an element of \(\ft^*_{-r}\)
that realises \(\phi\) on \(T_{r/2}\).
By \cite{yu:supercuspidal}*{Lemma 8.3}
(and the remark on base change at the top of
\cite{yu:supercuspidal}*{p.~597}),
the stabiliser of \(X^* + \pmb\ft(\unfield)^*_{(-r)+}\),
hence \textit{a fortiori} of \(X^*\),
in \(\bH(\unfield)\) is \(\bT(\unfield)\).
In particular, \(X^*\) is strongly regular semisimple.
Given this, and the fact that \(C_\bH(\bT) = \bT\)
\cite{borel:linear}*{Corollary 13.17(2)},
it is a straightforward check that,
for \(g \in \bH(\unfield)\),
we have that \(\Ad^*(g)X^*\) is defined over \field
if and only if
both the group \(\Int(g)\bT\)
and the map \map{\Int(g)}\bT{\Int(g)\bT}
are defined over \field.
Fix such an element \(g\), and put \(\bT' = \Int(g)\bT\).

The Moy--Prasad map for \bT (at depth \((r/2)+\))
is defined to be the unique isomorphism
\map
	{\mexp_{{(r/2)+}:{r+}}}
	{\ft_{{(r/2)+}:{r+}}}
	{T_{{(r/2)+}:{r+}}}
such that,
if \(Y \in \ft_{(r/2)+}\)
and \(\gamma \in \mexp_{{(r/2)+}:{r+}}(Y)\),
then
\[
\ord\bigl((\chi(\gamma) - 1) - \textup d\chi(Y)\bigr) > r
\]
for all characters \(\chi\) of \(\bT_\sepfield\)
\cite{adler:thesis}*{\S1.5, p.~11}.
Then
\begin{equation}
\tag{$*$}
\label{lem:SO:eq:mexp-equi}
\mexp'_{{(r/2)+}:{r+}}
\ldef \Int(g) \circ \mexp_{{(r/2)+}:{r+}} \circ \Ad(g)\inv
\end{equation}
satisfies the analogous property for \(\bT'\).
By the definition of `realise', if
\(Y \in \ft_{(r/2)+}\)
and \(\gamma \in \mexp_{{(r/2)+}:{r+}}(Y)\),
then
\begin{equation}
\tag{$**$}
\label{lem:SO:eq:conj-g}
(\phi \circ \Int(g)\inv)(\Int(g)\gamma)
= \phi(\gamma)
= \AddChar(\pair{X^*}Y)
= \AddChar(\pair{\Ad^*(g)X^*}{\Ad(g)Y}).
\end{equation}
Since \(\Ad(g)\) carries \(\ft_{(r/2)+}\) and \(\ft_{r+}\)
onto \(\ft'_{(r/2)+}\) and \(\ft'_{r+}\), respectively,
it follows from \eqref{lem:SO:eq:mexp-equi} and
\eqref{lem:SO:eq:conj-g}, and the definition of
`realise' again, that \(\Ad^*(g)X^*\) realises
\(\phi \circ \Int(g)\inv\) on \(\ft'_{(r/2)+}\).
In particular, since we have already observed that
the stabiliser of \(X^* + \pmb\ft(\unfield)_{(-r)+}\)
in \(\bH(\unfield)\) is \(\bT(\unfield)\), it follows that
\(\phi \circ \Int(g)\inv = \phi\)
(if and) only if
\(g \in \bT(\unfield)\).

In all, we have shown that
\anonmapto
	{(\Int(g)T, \phi \circ \Int(g)\inv)}
	{\Ad^*(g)X^*}
is a bijection from the \bH-stable conjugacy class of
\((T, \phi)\)
to \(\Ad^*(\bH(\unfield))X^* \cap \fg^*\)
such that, if \((T', \phi')\) has image \(X^{\prime\,*}\),
then \(\phi'\) is represented on \(\ft'_{(r/2)+}\) by
\(X^{\prime\,*}\);
in which case, in the notation of Definition \ref{defn:O-pi},
we have that \(\hat O^H_{\phi'} = \hat O^H_{X^{\prime\,*}}\).
Since the bijection is \(H\)-equivariant, the desired
equality follows.
\end{proof}

\begin{lem}
\label{lem:stab-approx}
If
\begin{itemize}
\item \bS is a \tamefield-split maximal \(G\)-torus,
\item \(\gamma \in S\),
\item \((\gamma_i)_{0 \le i < r}\) is
a normal \(r\)-approximation to \(\gamma\)
in \(G\)
\xcite{adler-spice:good-expansions}*
	{Definition \xref{defn:r-approx}},
and
\item \(g \in \bG(\sepfield)\) is such that
\(g\inv\sigma g \in \bS(\sepfield)\) for all
\(\sigma \in \Gamma\),
\end{itemize}
then \((\Int(g)\gamma_i)_{0 \le i < r}\) is
a normal \(r\)-approximation to \(\Int(g)\gamma\) in \(G\).
\end{lem}

\begin{rem}
\label{rem:stab-approx}
Under the hypotheses of Lemma \ref{lem:stab-approx},
\(\delta \ldef \Int(g)\gamma\) is \field-rational.
If \(\gamma\) is a strongly regular semisimple element of \(G\),
then the \field-rationality of \(\delta\) is
\emph{equivalent} to the hypothesis on \(g\).
\end{rem}

\begin{proof}
Recall that \(\gamma_{< r} = \prod_{0 \le i < r} \gamma_i\)
and \(\gamma_{\ge r} = \prod_{i \ge r} \gamma_i\)
\xcite{adler-spice:good-expansions}*
	{\S\xref{sec:normal}, p.~52}.

The group \(\bS' \ldef \Int(g)\bS\)
and the map \map{\Int(g)}\bS{\bS'}
are defined over \field.
In particular,
\(\bS'\) is a \tamefield-split, maximal \(G\)-torus,
and the set \(S'\) of \field-rational elements of \(\bS'\)
is precisely \(\Int(g)S\).

Fix an index \(i\) with \(0 \le i < r\).
Since \(\gamma_i \in Z(H) \subseteq S\)
\xcite{adler-spice:good-expansions}*
	{Remark \xref{rem:approx-facts-in-center}},
it follows that
\(\delta_i \ldef \Int(g)\gamma_i\) belongs to \(S'\).
If \(\delta_i \ne 1\), then,
because the root values
\xcite{adler-spice:good-expansions}*
	{Definition \xref{defn:root-value}}
of \(\gamma_i\) and \(\delta_i = \Int(g)\gamma_i\)
are the same,
we have that \(\delta_i\) is good of depth \(i\)
\xcite{adler-spice:good-expansions}*
	{Definition \xref{defn:good}}.

That is, \(\ul\delta \ldef (\delta_i)_{0 \le i < r}\) is a
good sequence
\xcite{adler-spice:good-expansions}*
	{Definition \xref{defn:funny-centralizer}}.
Since \((\gamma_i)_{0 \le i < r}\) is
a normal \(r\)-approximation to \(\gamma\)
\xcite{adler-spice:good-expansions}*
	{Definition \xref{defn:r-approx}},
we have that
\(\gamma_{\ge r} = \prod_{i \ge r} \gamma_i
	\in G_{y, r} \cap S = S_r\)
for some \(y \in \BB(\bS, \field)\)
(see \xcite{adler-spice:good-expansions}*
	{Definition \xref{defn:compatibly-filtered}
and Proposition \xref{prop:compatibly-filtered-tame-rank}}).
Thus
\(\delta_{\ge r} \ldef \prod_{i \ge r} \delta_i
= \Int(g)\gamma_{\ge r} \in \Int(g)S_r = S'_r
	\subseteq \CC G r(\ul\delta)_{y', r}\),
where \(y'\) is any point in
\(\BB(\bS', \field) \subseteq \BB(\CC\bG r(\ul\delta), \field)\);
so the result follows from
\xcite{adler-spice:good-expansions}*
	{Definition \xref{defn:r-approx}} again.
\end{proof}

It turns out that the original induction map \(\YuUp^G\)
is not quite suited to Reeder's conjectural construction of
L-packets
\cite{reeder:sc-pos-depth}*{\S6.6, p.~18},
precisely because of the appearance in Theorem \ref{thm:ratl}
of the character \(\varepsilon\noram(\phi)\).
In order to obtain stable character sums from Reeder's
construction, we must twist away this (quadratic) character.

\begin{defn}
\label{defn:twisted-cusp-ind}
Put
\(\twYuUp_T^G \phi
= \YuUp_T^G \bigl(
	\phi\dotm\varepsilon\noram(\phi)\inv
\bigr)\),
where \(\YuUp_T^G\) is as in Definition \ref{defn:cusp-ind}
and \(\varepsilon\noram(\phi)\)
is the character
\anonmapto\gamma{\varepsilon\noram(\phi, \gamma)}
of Definition \ref{defn:signs}.
\end{defn}

\begin{example}
\label{example:must-twist}
There are cases in which the collection of
representations
\[
\sett{\twYuUp_S^G \theta}
	{\((S, \theta)\) is a stable conjugate of \((T, \phi)\)}
\]
formed by our modified Yu-type construction is different
from the collection
\[
\sett{\YuUp_S^G \theta}
	{\((S, \theta)\) is a stable conjugate of \((T, \phi)\)}
\]
considered by Reeder.

For \(\bG = \operatorname{PGSp}_4\),
there exist stably \bG-conjugate, \unfield-split,
\field-anisotropic, maximal \(G\)-tori \bS and \bT
such that
\begin{itemize}
\item \(\Frob\) acts by inversion on both,
\item \(\BB(\bT, \field)\) contains a hyperspecial vertex
\(x\) of \(\BB(\bG, \field)\),
and
\item \(\BB(\bS, \field)\) contains a non-hyperspecial
vertex \(y\) of \(\BB(\bG, \field)\),
\end{itemize}
Let \(\ff_2/\ff\) be the quadratic extension of \ff,
so that \(\ff_\alpha = \ff_{\alpha'} = \ff_2\) for all roots
\(\alpha\) of \bT in \bG and \(\alpha'\) of \(\bT'\) in \bG,
and \(\ff_2^1\) the kernel of the norm map from
\(\ff_2\) to \ff.
Let \(\sset{\alpha, \beta}\) be a system of simple roots for
\bT in \bG, with \(\alpha\) short;
and \(\sset{\alpha', \beta'}\) the corresponding system of
simple roots for \bS in \bG.
Since all roots of \bS and \bT in \bG are symmetric,
we have that
\(\varepsilon\nosymm_{x, r/2}(G, T, \gamma) = 1\)
and
\(\varepsilon\nosymm_{y, r/2}(G, S, \gamma')\)
are trivial for all
\(r\),
\(\gamma \in T\),
and
\(\gamma' \in S\).
Let \(\phi\) be any \(G\)-generic character of \(T\)
\cite{yu:supercuspidal}*{\S9, p.~599},
and \(\theta\) the corresponding character of \(S\).
The common depth of \(\phi\) and \(\theta\) is
a positive integer \(r\).

If \(r\) is even, then
\(\Root(\bG, \bT)_{\textup{\(x, r/2\), symm, unram}}
= \Root(\bG, \bT)\)
and \(\Root(\bG, \bS)_{\textup{\(y, r/2\), symm, unram}}
= \sset{\pm\beta', \pm(2\alpha' + \beta')}\),
so that
\(\twYuUp_T^G \phi
= \YuUp_S^G\bigl(
	\phi(\sgn_{\ff_2^1} \circ \beta)\inv
\bigr)\)
and \(\twYuUp_S^G \theta = \YuUp_T^G \theta\).
If \(\gamma = \rho^\vee(t) \in T\), where
\(\rho^\vee = \frac3 2\alpha^\vee + 2\beta^\vee\) is the
half-sum of the positive co-roots of \bT in \bG (with
respect to the chosen system of simple roots)
and \(t\) is a generator of \(\ff_2^1\), then
\(\sgn_{\ff_2^1}(\beta(\gamma)) = \sgn_{\ff_2^1}(t) = -1\).

If \(r\) is odd, then
\(\Root(\bG, \bT)_{\textup{\(x, r/2\), symm, unram}}
= \emptyset\)
and \(\Root(\bG, \bS)_{\textup{\(y, r/2\), symm, unram}}
= \sset{\pm\alpha', \pm(\alpha' + \beta')}\),
so that
\(\twYuUp_T^G \phi = \YuUp_T^G \phi\)
and \(\twYuUp_S^G \theta
= \YuUp_S^G\bigl(
	\theta(\sgn_{\ff_2^1} \circ \beta')\inv
\bigr)\);
and, with the obvious notation,
\(\sgn_{\ff_2^1}(\beta'(\rho^{\prime\,\vee}(t))) = -1\).

Since the characters
\(\sgn_{\ff_2^1} \circ \beta\)
and \(\sgn_{\ff_2^1} \circ \beta'\) are trivial on the
subgroups of \(T\) and \(S\), respectively,
consisting of elements that admit a lift in
\(\Sp_4(\field)\),
we really must work on \(\operatorname{PGSp}_4(\field)\)
to see this phenomenon.
\end{example}

\begin{named}
	{Corollary \ref{cor:ratl} (to Theorem \ref{thm:ratl})}
\refstepcounter{thm}
\label{cor:ratl}
With the notation of Theorem \ref{thm:ratl},
\[
\Phi_{\twYuUp_T^G \phi}(\gamma)
= \sum_{\substack{
	(S, \theta) \in \CC G r(\gamma)\bslash G\dota(T, \phi) \\
	\gamma_{< r} \in S
}}
	\varepsilon\symmram(\theta, \gamma_{< r})\dotm
	\theta(\gamma_{< r})\dotm
	\tilde e(\theta, \gamma_{< r})
	\hat O^H_\theta(Y_{\ge r}).
\]
\end{named}

\begin{rem}
\label{rem:cor:ratl}
With the notation of Corollary \ref{cor:ratl}, the formula
there shows that \(\Phi_{\twYuUp_T^G \phi}(\gamma)\) vanishes unless
the \(G\)-orbit of \(\gamma_{< r}\) intersects \(T\); so,
for definiteness, we could require that actually
\(\gamma_{< r} \in T\), and then take \bT as the torus used to
define of \(Y_{\ge r}\) in the statement of
Theorem \ref{thm:ratl}.
\end{rem}

As we did for Theorem \ref{thm:ratl}, we re-capitulate in
Theorems \ref{thm:almost-stable} and \ref{thm:stable} below
the standing hypotheses of this section.  Remember that we
require \bG to satisfy
\xcite{adler-spice:good-expansions}*
	{Hypotheses \xref{hyp:conn-cent}
and \xref{hyp:good-weight-lattice}}.

\begin{thm}
\label{thm:almost-stable}
Suppose that the group \bG satisfies
Hypothesis \ref{hyp:stronger-mock-exp}
(in addition to
\xcite{adler-spice:good-expansions}*
	{Hypotheses \xref{hyp:conn-cent}
and \xref{hyp:good-weight-lattice}}).
Let \((T, \phi)\) be a depth-\(r\), toral, cuspidal pair,
and suppose that \bT is \unfield-split.
Suppose that \(\gamma = \gamma_{< r}\gamma_{\ge r}\) is
regular, semisimple, and \(r\)-approximable,
and that \(\gamma_{< r} \in T\),
and put \(Y_{\ge r} = \mexp_{\bT, x}\inv \gamma_{\ge r}\).
As in Lemma \ref{lem:SO}, write \(\bG\dota(T, \phi)\)
for the set of stable conjugates of \((T, \phi)\).
Then
\[
(-1)^{\rk_\field \bG}
\sum_{(S, \theta) \in G\bslash\bG\dota(T, \phi)}
	\Phi_{\twYuUp_S^G \theta}(\gamma)
\]
equals
\begin{align*}
\sum_{\substack{
	(S, \theta) \in \CC\bG r(\gamma)\bslash\bG\dota(T, \phi) \\
	\gamma_{< r} \in S
}}
	& \varepsilon\symmram(\theta, \gamma_{< r})\dotm
	\theta(\gamma_{< r})\dotm
	e(\theta, \gamma_{< r})\times{} \\
	& \qquad(-1)^{\rk_\field \CC\bG r(\gamma)}\dotm
	\smash{\widehat{SO}}^{\CC G r(\gamma)}_\theta(Y_{\ge r}),
\end{align*}
where \(\CC\bG r(\gamma)\bslash\bG\dota(T, \phi)\)
is the set of \(\CC\bG r(\gamma)\)-stable conjugacy classes
in \(\bG\dota(T, \phi)\).
\end{thm}

Note that the indexing sets for the two sums are different.

\begin{rem}
As in Remark \ref{rem:cor:ratl}, the assumption
that \(\gamma_{< r} \in T\) is just for convenience.
If the \(\bG(\unfield)\)-conjugacy class of \(\gamma_{< r}\)
does not intersect \(T\), then
the set of stable \bG-conjugates \((S, \theta)\)
of \((T, \phi)\) with \(\gamma_{< r} \in S\) is empty,
and
Corollary \ref{cor:ratl}
shows that the function
\((-1)^{\rk_\field \bG}\sum \Phi_{\twYuUp_S^G \theta}\)
of Theorem \ref{thm:almost-stable} vanishes at \(\gamma\).
\end{rem}

\begin{proof}
Put \(\bH = \CC\bG r(\gamma)\).
By Corollary \ref{cor:ratl}
and Proposition \ref{prop:stable-sign},
\[
(-1)^{\rk_\field \bG}
\sum_{(S, \theta) \in G\bslash\bG\dota(T, \phi)}
	\Phi_{\twYuUp_S^G \theta}(\gamma)
\]
is equal to the summation
\begin{align*}
& \sum_{(S, \theta) \in G\bslash\bG\dota(T, \phi)}
	\:
	\sum_{\substack{
		(S', \theta') \in H\bslash G\dota(S, \theta) \\
		\gamma_{< r} \in S'
	}} & F(\theta') \\
= & \sum_{\substack{
		(S', \theta') \in H\bslash\bG\dota(T, \phi) \\
		\gamma_{< r} \in S'
}} & F(\theta') \\
= & \sum_{(S, \theta) \in \bH\bslash\bG\dota(T, \phi)}
	\:
	\sum_{\substack{
		(S', \theta') \in H\bslash\bH\dota(S, \theta) \\
		\gamma_{< r} \in S'
	}} & F(\theta'),
\end{align*}
where
\(\bH\bslash\bG\dota(T, \phi)\) is the set of
stable \bG-conjugacy classes in \(\bG\dota(T, \phi)\),
\(\bH\dota(S, \theta)\) is the set of
stable \bH-conjugates of \((S, \theta)\),
and
\[
F(\theta')
\ldef
\varepsilon\symmram(\theta', \gamma_{< r})\dotm
\Phi_{\theta'}(\gamma_{< r})\dotm
e(\theta', \gamma_{< r})\dotm
(-1)^{\rk_\field \bH}
\hat O^H_{\theta'}(Y_{\ge r})
\]
for \((S', \theta') \in \bH\dota(S, \theta)\).
Note that, in the above notation,
the condition \(\gamma_{< r} \in S'\),
the number \(\Phi_{\theta'}(\gamma_{< r})\),
and (by Lemma \ref{lem:stable-sign})
the numbers
\(\varepsilon\symmram(\theta', \gamma_{< r})\)
and \(e(\theta', \gamma_{< r})\) depend only on
\((S, \theta)\), not \((S', \theta')\);
and that, by
Definitions \ref{defn:disc} and \ref{defn:normal-harm},
the function \(\redD_S\) is identically \(1\),
so that \(\Phi_\theta = \theta\).
The result now follows from Lemma \ref{lem:SO}.
\end{proof}

\begin{thm}
\label{thm:stable}
Suppose that \bG satisfies
\cite{debacker-reeder:depth-zero-sc}*
	{Restriction 12.4.1(2)}
(in addition to
\xcite{adler-spice:good-expansions}*
	{Hypotheses \xref{hyp:conn-cent}
and \xref{hyp:good-weight-lattice}}).
Let \((T, \phi)\) be a positive-depth, toral, cuspidal pair,
and suppose that \bT is \unfield-split.
Then
\[
\Theta\textsup{st}_{\twYuUp_T^G \phi} \ldef
\sum_{(S, \theta) \in G\bslash\bG\dota(T, \phi)}
	\Theta_{\twYuUp_S^G \theta}
\]
is stable on the set of
\(r\)-approximable and strongly regular semisimple
elements of \(G\).
That is, it is constant on each stable conjugacy class
(see \cite{kottwitz:ratl-conj}*{\S3, p.~788}
and
\cite{debacker-reeder:depth-zero-sc}*{\S2.9, p.~817})
in that set.
\end{thm}

\begin{rem}
By \xcite{adler-spice:good-expansions}*
	{Lemma \xref{lem:simult-approx}}, under mild
conditions, every element of a tame torus is
\(r\)-approximable.
Of course, if \(p\) is sufficiently large, then every
semisimple element of \(G\) lies in a tame torus.
\end{rem}

\begin{proof}
By \cite{debacker-reeder:depth-zero-sc}*{Lemma B.0.3}, the
exponential map \mexp converges on \(\fg(\unfield)_{0+}\);
so we may, and do, take the maps \(\mexp_{\bT, \ox}\) of
Hypothesis \ref{hyp:stronger-mock-exp} to be restrictions 
of \mexp.

Let \(\gamma = \gamma_{< r}\gamma_{\ge r}\) be an
\(r\)-approximable and strongly regular semisimple element
of \(G\).
By Theorem \ref{thm:almost-stable},
\(\Theta\textsup{st}_{\twYuUp_S^G \theta}(\gamma) = 0\)
unless the \(\bG(\unfield)\)-orbit of \(\gamma_{< r}\)
intersects \(T\).

Thus, we may, and do, suppose that \(\gamma_{< r} \in T\),
and then show that
\begin{equation}
\tag{$*$}
\label{thm:stable:eq:stable}
\Theta\textsup{st}_{\twYuUp_S^G \theta}(\gamma)
= \Theta\textsup{st}_{\twYuUp_S^G \theta}(\Int(g)\gamma)
\end{equation}
for any element \(g \in \bG(\unfield)\) such that
\(\Int(g)\gamma\) is defined over \field,
\ie, such that \(g\inv\Frob g \in C_\bG(\gamma)(\unfield)\).

Given such an element \(g\), put
\(\gamma' = \Int(g)\gamma\)
and \(\bT' = \Int(g)\bT\).
As usual, note that \(\bT'\) is defined over \field,
and \(\gamma' \in T'\).
By Lemma \ref{lem:stab-approx}, the element \(\gamma'\) is
also \(r\)-approximable, with
\(\gamma'_{< r} = \Int(g)\gamma_{< r}\)
and
\(\gamma'_{\ge r} = \Int(g)\gamma_{\ge r}\).
Put \(Y'_{\ge r} = \mexp\inv \gamma'_{\ge r}\).
Put \(\bJ = \CC\bG r(\gamma') = \Int(g)\bH\)
\xcite{adler-spice:good-expansions}*
	{Corollary \xref{cor:compare-centralizers}}.
Exactly as in \cite{debacker-reeder:depth-zero-sc}*
	{\S11.1, p.~861},
there is a bijection
\[
\map
	{\iota_g}
	{\bH\bslash\set
		{(S, \theta) \in \bG\dota(T, \phi)}
		{\gamma_{< r} \in S}
	}
	{\bJ\bslash\set
		{(S', \theta') \in \bG\dota(T, \phi)}
		{\gamma'_{< r} \in S'}
	},
\]
where, as usual, \(\bG\dota(T, \phi)\) is the set of stable
\bG-conjugates of \((T, \phi)\), and the notations
\(\bH\bslash\cdots\) and \(\bJ\bslash\cdots\)
stand for the collections of stable \bH- and \bJ-conjugacy
classes, respectively, in the appropriate sets.
If \(\iota_g(S, \theta) = (S', \theta')\), then
	\begin{itemize}
	\item by construction, we have that
\(\theta(\gamma_{< r}) = \theta'(\gamma'_{< r})\);
	\item by Lemma \ref{lem:stable-sign}, we have that
\(\varepsilon\symmram(\theta, \gamma_{< r})
= \varepsilon\symmram(\theta', \gamma'_{< r})\)
and
\(e(\theta, \gamma_{< r})
= e(\theta', \gamma'_{< r})\);
	and
	\item by
\cite{debacker-reeder:depth-zero-sc}*{Lemma 12.2.3},
we have that
\[
(-1)^{\rk_\field \bJ}\dotm
\smash{\widehat{SO}}^J_{\theta'}(Y'_{\ge r})
= (-1)^{\rk_\field \bH}\dotm
\smash{\widehat{SO}}^H_\theta(Y_{\ge r}).
\]
(See \cite{debacker-reeder:depth-zero-sc}*{Example 12.2.1}.)
	\end{itemize}
The equality \eqref{thm:stable:eq:stable} now follows from
Theorem \ref{thm:almost-stable}.
\end{proof}

\end{document}